\documentclass[11pt,a4paper]{article}
\usepackage{amsmath}
\usepackage[a4paper]{geometry}
\usepackage{amssymb,latexsym,amsmath,amsfonts,amsthm}
\usepackage{graphicx}
\usepackage{algorithm}
\usepackage{algorithmic}
 \usepackage{dsfont}
\usepackage{mathrsfs}
\usepackage{epsfig}
\usepackage{booktabs}
\usepackage{mathtools}
\usepackage{comment,verbatim}
\usepackage{overpic}
\usepackage{hyperref}
\usepackage{bbm}
\usepackage[toc,page]{appendix}

\renewcommand{\vec}{\mathbf}

\newtheorem{theorem}{Theorem}[section]
\newtheorem{lemma}[theorem]{Lemma}

\newtheorem{corollary}[theorem]{Corollary}

\theoremstyle{definition}

\newtheorem{notation}[theorem]{Notation}

\theoremstyle{definition}

\theoremstyle{remark}

\newtheorem{remark}[theorem]{Remark}

\numberwithin{equation}{section}

\usepackage{color}

\usepackage[normalem]{ulem}

\def\XXint#1#2#3{{
\setbox0=\hbox{$#1{#2#3}{\int}$}
\vcenter{\hbox{$#2#3$}}\kern-.5\wd0}}

\begin{document}
\title{A unified framework for asymptotic analysis and computation of finite Hankel transform} 
\author{Haiyong Wang\footnotemark[2]~\footnotemark[3]}

\maketitle
\renewcommand{\thefootnote}{\fnsymbol{footnote}}

\footnotetext[2]{School of Mathematics and Statistics, Huazhong
University of Science and Technology, Wuhan 430074, P. R. China.
E-mail: \texttt{haiyongwang@hust.edu.cn}}

\footnotetext[3]{Hubei Key Laboratory of Engineering Modeling and
Scientific Computing, Huazhong University of Science and Technology,
Wuhan 430074, China.}

\begin{abstract}
In this paper we present a unified framework for asymptotic analysis
and computation of the finite Hankel transform. This framework
enables us to derive asymptotic expansions of the transform,
including the cases where the oscillator has zeros and stationary
points. As a consequence, two efficient and affordable methods for
computing the transform numerically are developed and a detailed
analysis of their asymptotic error estimate is carried out.
Numerical examples are
provided to confirm our analysis. 
\end{abstract}


\vspace{0.05in}


\section{Introduction}\label{sec:introduction}
The finite Hankel transform of a function $f: [a,b]\rightarrow
\mathbb{R}$ with a general oscillator $g$ is defined by
\begin{align}\label{eq:Hankel}
\mathcal{H}_{\nu}[f] := \int_{a}^{b} f(x) J_{\nu}(\omega g(x)) dx,
\end{align}
where $J_{\nu}(x)$ is the Bessel function of the first kind of order
$\nu$ and $\omega>0$ is a large parameter. This transform plays an
important role in many physical problems such as boundary value
problems formulated in cylindrical coordinate
\cite{poularikas2010transforms}, time dependent Schr\"{o}dinger
equation \cite{bisseling1985fast}, electromagnetic scattering
\cite{bao2005fast} and geometrical acoustics
\cite{boucher2017phased}.

Closed forms for the finite Hankel transform are rarely available in
general and, except in a few special cases, numerical methods are
required to evaluate the transform accurately. However, the
calculation with traditional quadrature rules may be prohibitively
expensive due to the oscillatory character of the Bessel function.
Roughly speaking, when $\omega\gg1$, a fixed number of quadrature
nodes per oscillation is required to obtain an acceptable level of
accuracy, which makes the complexity grows linearly with $\omega$.

In the past decade, the evaluation of highly oscillatory integrals
has been received considerable attention. In particular, for
Fourier-type integrals of the form
\begin{align}\label{eq:Fourier}
\int_{a}^{b} f(x) e^{i\omega g(x)} dx,
\end{align}
several efficient methods such as asymptotic method
\cite{iserles2004quadrature,iserles2005efficient}, Filon-type
methods
\cite{iserles2004quadrature,iserles2005efficient,xiang2007efficient},
Levin-type methods \cite{levin1982procedure,olver2006moment},
numerical steepest descent methods
\cite{deano2009complex,huybrechs2006evaluation}, complex Gaussian
quadrature \cite{asheim2014gauss,asheim2013complex} and their
different combinations or adaptive methods
\cite{dominguez2011stability,gao2017efficient,huybrechs2012super,zhao2017adaptive}
have been developed to compute the value of \eqref{eq:Fourier} with
low computational cost. Among these methods, {\it asymptotic method}
plays a fundamental role in clarifying these critical points, e.g.,
endpoints and stationary points, to the value of the integral and
serves as a starting point of designing other efficient methods. We
refer to the recent monograph \cite{deano2017highly} for a
comprehensive survey.

In recent years, there has been a growing  
interest in studying numerical quadrature of the finite Hankel
transform \eqref{eq:Hankel}; see, for instance,
\cite{chen2012numerical,chen2015numerical,levin1996fast,olver2007numerical,piessens1983modified,wang2011asymptotic,xiang2007numerical,xiang2011clenshaw,xiang2008numerical,xu2016efficient}.
All these works focus on the simplest oscillator $g(x)=x$ and
asymptotic analysis has been carried out either by using the
following differential equation
\cite{olver2007numerical,xiang2008numerical}
\begin{align}\label{eq:VectODE}
\mathbf{w}{'}(x) = \mathbf{A}(\omega,x) \mathbf{w}(x),
\end{align}
where
\[
\mathbf{w}(x) = \binom{J_{\nu-1}(\omega x)}{J_{\nu}(\omega x)},
\quad \mathbf{A}(\omega,x) = \left(
                                                                     \begin{array}{cc}
                                                                       \frac{\nu-1}{x} &  -\omega        \\
                                                                       \omega          & -\frac{\nu}{x}  \\
                                                                     \end{array}
                                                                   \right),
\]
or by using the well-known identity
\cite{wong1976error,xiang2010fast}
\begin{align}\label{eq:DE}
\frac{d}{dz} (z^{\nu+1} J_{\nu+1}(z)) = z^{\nu+1} J_{\nu}(z), \quad
z,\nu\in \mathbb{C}.
\end{align}
For the Hankel transform with a general oscillator $g(x)$, very few
studies have been conducted so far in the literature, except in the
very special case where $g^{(j)}(a)=0$ for $j=0,\ldots,r$ and $r$ is
some positive integer \cite{xiang2010fast}.

In this work, we aim to provide a complete asymptotic analysis and
the construction of affordable quadrature rules for the finite
Hankel transform with a general oscillator. Intuitively speaking,
the main contribution to the value of the transform comes from the
following three types of {\it critical points}:
\begin{enumerate}
\item Endpoints of the interval of integration, i.e., $x=a$ and
$x=b$;

\item Zeros of the oscillator $\xi\in[a,b]$ where $g(\xi)=0$;

\item Stationary points of the oscillator $\zeta\in[a,b]$ where $g{'}(\zeta) =
0$.
\end{enumerate}
For each type of critical point listed above, we will explore the
asymptotic expansion of the transform. Note that the asymptotic
analysis based on \eqref{eq:VectODE} fails when the oscillator
$g(x)$ has zeros or stationary points on the integration interval.
We will use the identity \eqref{eq:DE} and show that it is well
suited for obtaining the asymptotic expansion of the finite Hankel
transform. This is not trivial and requires more in-depth
theoretical analysis. With these expansions in hand, we further
develop two methods for computing the transform numerically and a
rigorous analysis of their error is discussed.

The main results of the present paper can be summarized as follows:
\begin{itemize}
\item We provide a unified framework for asymptotic analysis
of the finite Hankel transform. It enables us to obtain asymptotic
expansions of the transform, especially in the presence of
stationary points.

\item Two methods for computing the transform are
discussed and a detailed error analysis is performed. Compared to
the method in \cite{xiang2010fast}, the method proposed here has the
advantage of avoiding the use of diffeomorphism transformations
which may be inconvenient in practice.

\item We show that our analysis can be extended to the multivariate
setting.
\end{itemize}


The structure of this paper is as follows. In the next section, we
start with the ideal case where the oscillator $g(x)$ is free of
stationary points, i.e., $g{'}(x)\neq 0$ for $x\in[a,b]$. We
distinguish our analysis into two cases according to $g(x)$ has a
zero on $[a,b]$ or not. In addition, two methods for evaluating the
transform numerically are developed and the behavior of their error
is discussed thoroughly. In section \ref{sec:Stationary} we deal
with the case where $g(x)$ has a stationary point on the integration
interval. The extension of our results is discussed in section
\ref{sec:MultiHankel} and some concluding remarks are presented in
section \ref{sec:con}.

\begin{notation}
Throughout the paper, we denote the generalized moments of the
Hankel transform $\mathcal{H}_{\nu}[f]$ by
\begin{align}
M_j(\zeta,\nu,\omega) = \mathcal{H}_{\nu}[(x-\zeta)^j] =
\int_{a}^{b} (x-\zeta)^j J_{\nu}(\omega g(x)) dx,
\end{align}
where $\zeta\in[a,b]$ and $j\geq0$. When $j=0$, note that
$M_0(\zeta,\nu,\omega)$ does not depend on $\zeta$ anymore. Instead,
we define $M(\nu,\omega) = M_0(\zeta,\nu,\omega)$.
\end{notation}

\section{The case with no stationary points}\label{sec:FreeStationary}
In this section we first consider the ideal case that the oscillator
$g(x)$ has no stationary points in the integration interval $[a,b]$,
i.e., $g{'}(x)\neq 0$ for $x\in [a,b]$. We commence our analysis
from the case that $g(x)\neq0$ for $x\in[a,b]$.

\subsection{The case $g(x)\neq0$ for $x\in[a,b]$}
We now state the first main result on the asymptotic expansion of
$\mathcal{H}_{\nu}[f]$.
\begin{theorem}\label{thm:AsympCaseOne}
Let $f(x),g(x)\in C^{\infty}[a,b]$ and $g(x)\neq0$, $g{'}(x)\neq0$
for $x\in[a,b]$. Then, for $\omega\rightarrow\infty$,
\begin{align}\label{eq:AsymCaseOne}
\mathcal{H}_{\nu}[f] \sim  - \sum_{k=1}^{\infty}
\frac{1}{(-\omega)^{k}} \left\{ \frac{\sigma_{k-1}[f](b)}{g{'}(b)}
J_{\nu+k}(\omega g(b)) - \frac{\sigma_{k-1}[f](a)}{g{'}(a)}
J_{\nu+k}(\omega g(a))  \right\},
\end{align}
where $\sigma_{k}[f](x)$ are defined by
\begin{align}\label{def:SigmaCase1}
\sigma_0[f](x) = f(x), \quad \sigma_k[f](x) = g(x)^{\nu+k}
\frac{d}{dx}\left[\frac{\sigma_{k-1}[f](x)}{g(x)^{\nu+k} g{'}(x)}
\right], \quad k\geq 1.
\end{align}
\end{theorem}
\begin{proof}
We shall first prove by induction on $m$ the following identity
\begin{align}\label{eq:expansion one}
\mathcal{H}_{\nu}[f] &= - \sum_{k=1}^{m} \frac{1}{(-\omega)^{k}}
\left\{ \frac{\sigma_{k-1}[f](b)}{g{'}(b)} J_{\nu+k}(\omega g(b)) -
\frac{\sigma_{k-1}[f](a)}{g{'}(a)}
J_{\nu+k}(\omega g(a))  \right\} \nonumber \\
&~~~~~~~~~~ + \frac{1}{(-\omega)^m}
\mathcal{H}_{\nu+m}[\sigma_{m}[f]].
\end{align}
When $m=0$, it is trivial. 
Now suppose \eqref{eq:expansion one} is true for $m$, we wish to
prove it for $m+1$.

By taking $z = \omega g(x)$ and replacing $\nu$ by $\nu+m$ in
\eqref{eq:DE}, we obtain
\begin{align}\label{eq:DerivOscillator}
\frac{d}{dx}\left[ g(x)^{\nu+m+1} J_{\nu+m+1}(\omega g(x)) \right] =
\omega g{'}(x) g(x)^{\nu+m+1} J_{\nu+m}(\omega g(x)).
\end{align}
Inserting this into $\mathcal{H}_{\nu+m}[\sigma_{m}[f]]$ and using
integration by parts we find
\begin{align*}
\mathcal{H}_{\nu+m}[\sigma_m[f]] &= \int_{a}^{b} \sigma_{m}[f](x)
J_{\nu+m}(\omega g(x)) dx \nonumber \\
&= \frac{1}{\omega} \int_{a}^{b} \frac{\sigma_{m}[f](x)}{
g(x)^{\nu+m+1} g{'}(x)}
\frac{d}{dx}\left[ g(x)^{\nu+m+1} J_{\nu+m+1}(\omega g(x)) \right] dx \nonumber \\
&= \frac{1}{\omega} \left[ \frac{\sigma_{m}[f](b)}{g{'}(b)}
J_{\nu+m+1}(\omega g(b)) - \frac{\sigma_{m}[f](a)}{g{'}(a)}
J_{\nu+m+1}(\omega g(a)) \right] \nonumber  \\
&~~~~~~~ -\frac{1}{\omega} \int_{a}^{b} g(x)^{\nu+m+1}
\frac{d}{dx}\left[\frac{\sigma_{m}[f](x)}{g(x)^{\nu+m+1} g{'}(x)}
\right]
J_{\nu+m+1}(\omega g(x)) dx \\
&= \frac{1}{\omega} \left[ \frac{\sigma_{m}[f](b)}{g{'}(b)}
J_{\nu+m+1}(\omega g(b)) - \frac{\sigma_{m}[f](a)}{g{'}(a)}
J_{\nu+m+1}(\omega g(a)) \right] \nonumber  \\
&~~~~~~~ -\frac{1}{\omega} \mathcal{H}_{\nu+m+1}[\sigma_{m+1}[f]],
\end{align*}
which, together with \eqref{eq:expansion one}, implies that
\begin{align}
\mathcal{H}_{\nu}[f] &= - \sum_{k=1}^{m+1} \frac{1}{(-\omega)^{k}}
\left\{ \frac{\sigma_{k-1}[f](b)}{g{'}(b)} J_{\nu+k}(\omega g(b)) -
\frac{\sigma_{k-1}[f](a)}{g{'}(a)}
J_{\nu+k}(\omega g(a))  \right\} \nonumber \\
&~~~~~~~~~~ + \frac{1}{(-\omega)^{m+1}}
\mathcal{H}_{\nu+m+1}[\sigma_{m+1}[f]].  \nonumber
\end{align}
Hence the identity \eqref{eq:expansion one} is also true for $m+1$.
By induction, we conclude that \eqref{eq:expansion one} is true for
all integers $m\geq0$.

Finally, the desired result \eqref{eq:AsymCaseOne} follows by
letting $m\rightarrow\infty$ in \eqref{eq:expansion one}. This
completes the proof.
\end{proof}

As a direct consequence of this theorem, we obtain the following
expansion for the simplest case $g(x)=x$.
\begin{corollary}\label{cor:IdealAsym}
In the case $g(x)=x$ and $b>a>0$ or $a<b<0$, then 
\begin{align}\label{eq:IdealAsym}
\mathcal{H}_{\nu}[f] 
\sim -
\sum_{k=1}^{\infty} \frac{1}{(-\omega)^{k}} \bigg\{
\sigma_{k-1}[f](b) J_{\nu+k}(\omega b) - \sigma_{k-1}[f](a)
J_{\nu+k}(\omega a) \bigg\},
\end{align}
where
\begin{align}
\sigma_{0}[f](x) = f(x), \qquad \sigma_{k}[f](x) &= x^{\nu+k}
\frac{d}{dx}\left[\frac{\sigma_{k-1}[f](x)}{x^{\nu+k}} \right],
\quad k\geq1. \nonumber
\end{align}
\end{corollary}

What is the dependence of $\sigma_{k}[f](x)$ on $f(x)$ and its
derivatives? After some algebraic manipulations, the first few terms
are given explicitly by
\begin{align}\label{eq:ExpFormSigma}
\sigma_0[f](x) &= f, \nonumber \\
\sigma_1[f](x) &= \frac{1}{g{'}} f{'} - \frac{g g{''} + (\nu+1)
(g{'})^2}{g (g{'})^2} f, \\
\sigma_2[f](x) &= \frac{1}{(g{'})^2} f{''} - \left(
\frac{3g{''}}{(g{'})^3} + \frac{2\nu+3}{g{'}g} \right) f{'}
\nonumber
\\
&~~~~~ - \left( \frac{g{'''}}{(g{'})^3} -
\frac{3(g{''})^2}{(g{'})^4} - \frac{(2\nu+3)g{''}}{(g{'})^2 g} -
\frac{(\nu+1)(\nu+3)}{g^2} \right) f,  \nonumber
\end{align}
and it is easy to verify that each $\sigma_k[f](x)$ is a linear
combination of $f(x),f'(x),\ldots,f^{(k)}(x)$ with coefficients are
rational functions of $g(x)$ and its derivatives.

If we truncate \eqref{eq:AsymCaseOne} after the first $m$ terms,
this results in an asymptotic method
\begin{align}\label{def:AsympMethod}
Q_m^{A}[f] = -\sum_{k=1}^{m} \frac{1}{(-\omega)^{k}} \left\{
\frac{\sigma_{k-1}[f](b)}{g{'}(b)} J_{\nu+k}(\omega g(b)) -
\frac{\sigma_{k-1}[f](a)}{g{'}(a)} J_{\nu+k}(\omega g(a)) \right\}.
\end{align}
A few remarks on the asymptotic method are in order.
\begin{remark}
By using the dependence relation \eqref{eq:ExpFormSigma}, it is easy
to verify that $Q_m^{A}[f]$ is a linear combination of the values
$f(a),f{'}(a),\ldots,f^{(m-1)}(a)$ and
$f(b),f{'}(b),\ldots,f^{(m-1)}(b)$.
\end{remark}

\begin{remark}
Under the assumption of neglecting the complexity of computing
Bessel functions $J_{\nu}(x)$\footnote{They can be computed
efficiently by using modern mathematical programming languages such
as Maple and Matlab.}, the asymptotic method $Q_m^{A}[f]$ can be
achieved at a low cost. For example, when $m=1$, we have
\[
Q_1^{A}[f] = \frac{1}{\omega} \left\{ \frac{f(b)}{g{'}(b)}
J_{\nu+1}(\omega g(b)) - \frac{f(a)}{g{'}(a)} J_{\nu+1}(\omega g(a))
\right\}.
\]
Clearly, we only need to operate on the values
$f(a),f(b),g{'}(a),g{'}(b)$. Moreover, as shown in Theorem
\ref{thm:ErrorAsymp} below, the accuracy of $Q_{m}^{A}[f]$ improves
substantially as $\omega$ increases.
\end{remark}

\begin{theorem}\label{thm:ErrorAsymp}
Under the same assumptions as in Theorem \ref{thm:AsympCaseOne}, we
have
\begin{align}
\mathcal{H}_{\nu}[f] - Q_m^{A}[f] =
\mathcal{O}(\omega^{-m-\frac{3}{2}}), \quad \omega \rightarrow
\infty.
\end{align}
\end{theorem}
\begin{proof}
From \eqref{eq:AsymCaseOne}, we obtain
\begin{align*}
\mathcal{H}_{\nu}[f] - Q_m^{A}[f] \sim   - \sum_{k=m+1}^{\infty}
\frac{1}{(-\omega)^{k}} \left\{ \frac{\sigma_{k-1}[f](b)}{g{'}(b)}
J_{\nu+k}(\omega g(b)) - \frac{\sigma_{k-1}[f](a)}{g{'}(a)}
J_{\nu+k}(\omega g(a)) \right\}.
\end{align*}
Note that $J_{\nu}(z) = \mathcal{O}(z^{-\frac{1}{2}})$ as
$z\rightarrow\infty$ \cite[Eqn.~10.17.2]{olver2010nist}. The desired
result follows.
\end{proof}

In Figure \ref{fig:Asymptotic method} we show an example where $f(x)
=\cos x$, $g(x)=x^2+x$, $\nu=1$ and $[a,b]=[1,2]$. We display the
error of $Q_{m}^{A}[f]$ in the left panel and the error scaled by
$\omega^{m+\frac{3}{2}}$ in the right panel. It is clear to see that
the accuracy of $Q_{m}^{A}[f]$ improves as $\omega$ increases and
the decay rate of the error of $Q_{m}^{A}[f]$ is
$\mathcal{O}(\omega^{-m-\frac{3}{2}})$, which confirms the result of
Theorem \ref{thm:ErrorAsymp}.

\begin{figure}[ht]
\centering
\includegraphics[width=7.2cm,height=6cm]{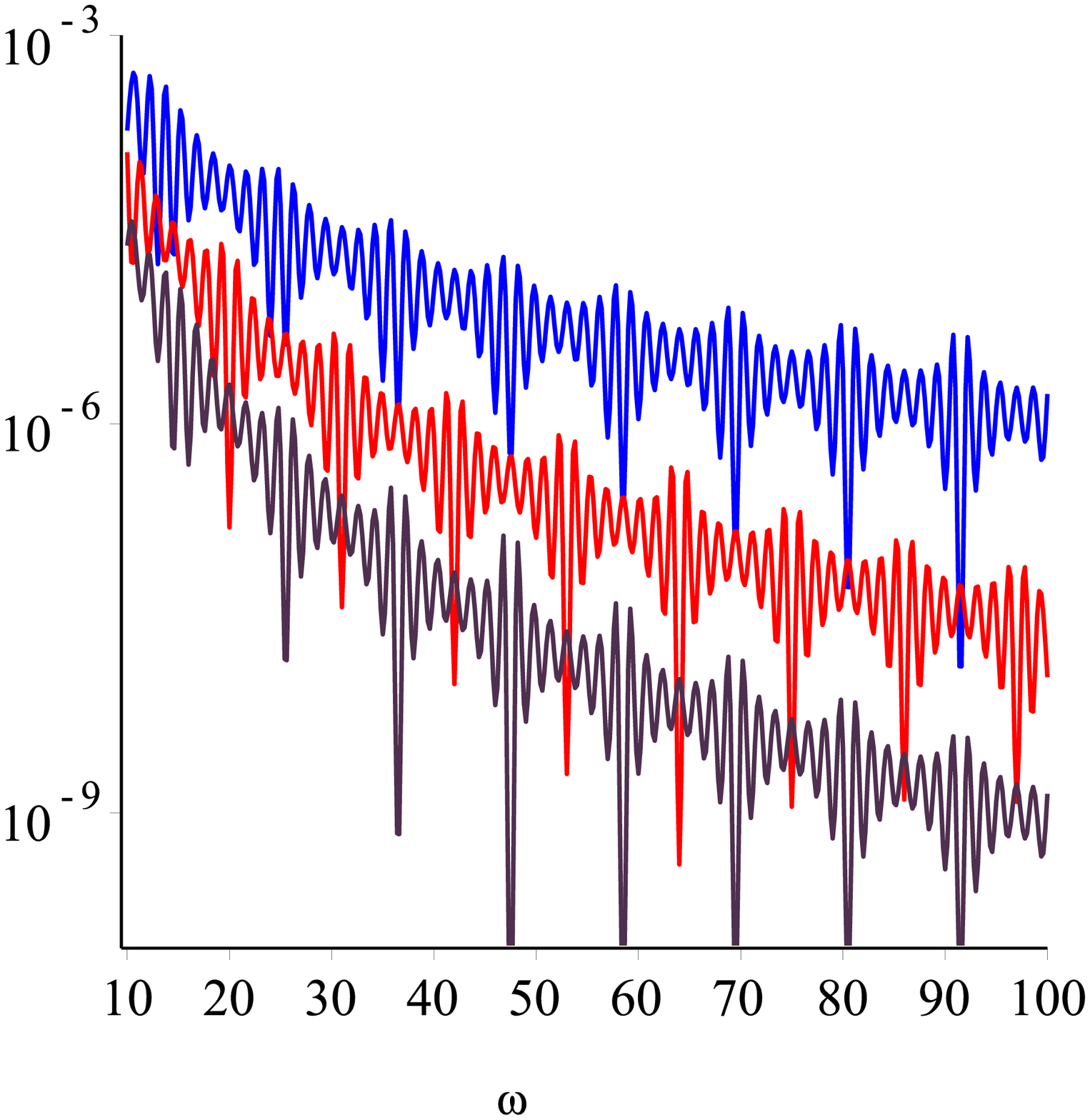}~~
\includegraphics[width=7.2cm,height=6cm]{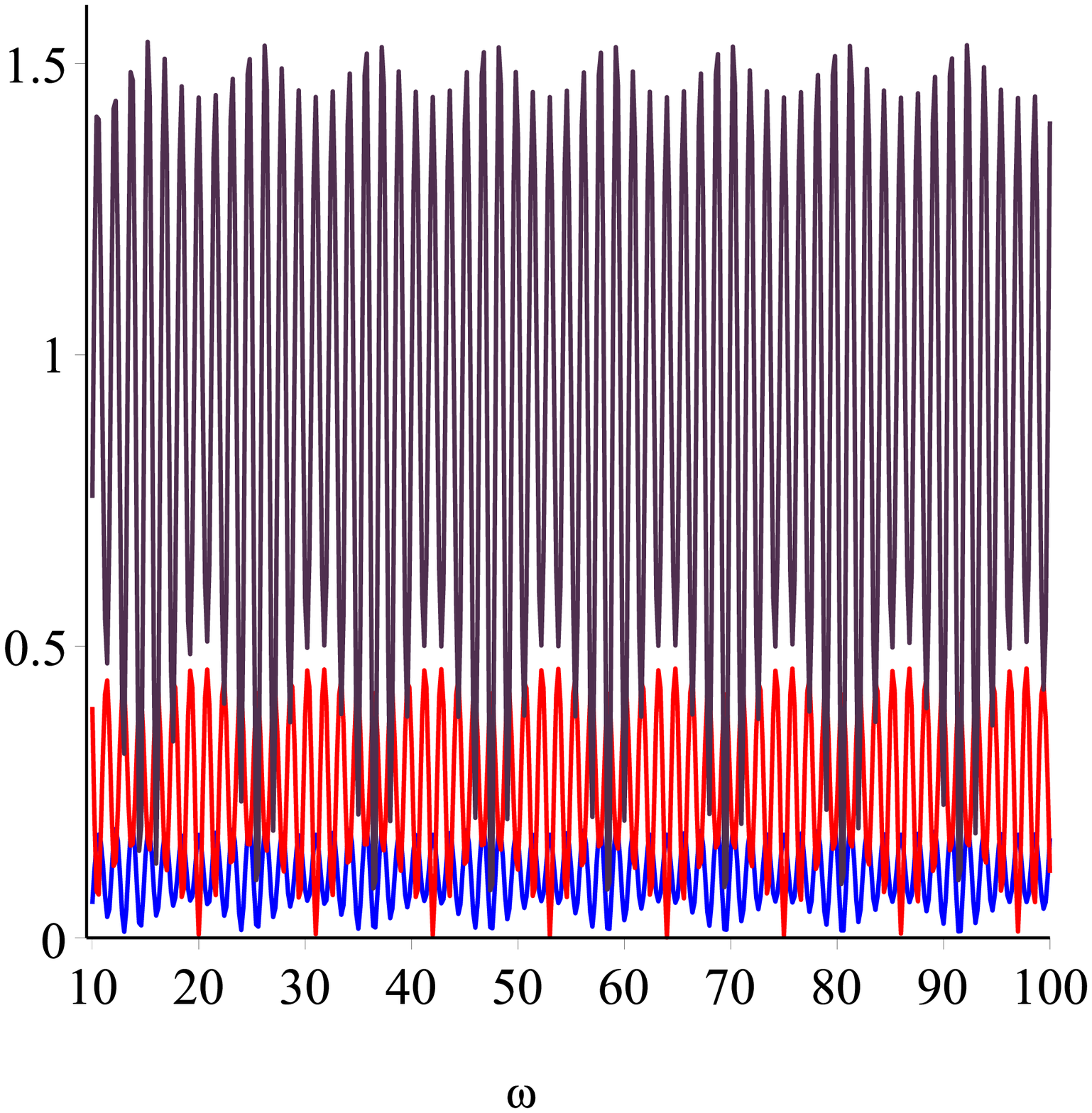}
\caption{The error of $Q_m^{A}[f]$ (left) and the error scaled by
$\omega^{m+\frac{3}{2}}$ (right). 
The blue, red, violet curves correspond to $m=1,2,3$, respectively.}
\label{fig:Asymptotic method}
\end{figure}

A major disadvantage of the asymptotic method is that the error is
essentially uncontrollable for a fixed $\omega$. To overcome this
drawback, we will 
propose a new Filon-type method which achieves the same error
estimate as the asymptotic method, whilst significantly improving
the accuracy. Before proceeding, let us establish a helpful lemma.
\begin{lemma}\label{lem:space}
Assume that $g(x)\in C^{\infty}[a,b]$ and $g{'}(x)\neq0$ for
$x\in[a,b]$. Let
\[
\varphi_k(x) = g{'}(x) g(x)^k, \quad k\geq0,
\]
and set
\[
\mathcal{E} = \left\{v(x)~\bigg|~ v(x) = \sum_{k=0}^{n-1} c_k
\varphi_k(x),~ c_k\in \mathbb{R} \right\}.
\]
Then, $\mathcal{E}$ forms an extended Chebyshev space.
\end{lemma}
\begin{proof}
With the assumption $g(x)\in C^{\infty}[a,b]$ we see that
$\varphi_k(x)\in C^{\infty}[a,b]$. For any $v(x)\in \mathcal{E}$, we
have, by setting $y=g(x)$, that
\[
v(x) = \sum_{k=0}^{n-1} c_k \varphi_k(x) = g{'}(x) \sum_{k=0}^{n-1}
c_k y^k.
\]
On the one hand, we can see that $v(x)\equiv0$ implies that $c_k=0$,
$k=0,\ldots,n-1$, and therefore
$\{\varphi_0(x),\ldots,\varphi_{n-1}(x)\}$ forms a basis of
$\mathcal{E}$. On the other hand, we can see that $v(x)$ has at most
$n-1$ zeros in $[a,b]$ counting multiplicities. From \cite[Theorem
2.33]{schumaker2007spline}, it follows that $\mathcal{E}$ is an
extended Chebyshev space. This completes the proof.
\end{proof}

We are now ready to construct a new Filon-type method, which is a
modification of the standard method presented in
\cite{iserles2004quadrature,iserles2005efficient}. Let
$a=x_0<x_1<\cdots<x_d=b$ be a set of distinct nodes with
multiplicities $m_0,m_1,\ldots,m_d$ and $\sum_{k=0}^{d} m_k = n$.
From \cite[Theorem 9.9]{schumaker2007spline} we know that there
exists a unique function $p(x) =\sum_{k=0}^{n-1} c_k \varphi_k(x)
\in \mathcal{E}$ such that
\begin{align}\label{eq:HermiteBirkhoff}
p^{(j-1)}(x_k) = f^{(j-1)}(x_k),
\end{align}
for all $j=1,\ldots,m_k$ and $k=0,\ldots,d$. The modified Filon-type
method is defined by
\begin{align}\label{eq:ModifiedFilon}
Q^{F}[f] := \mathcal{H}_{\nu}[p] = \sum_{k=0}^{n-1} c_k
\mu_{k,\nu}^{\omega},
\end{align}
where we have introduced the modified moments
$\mu_{k,\nu}^{\omega}=\mathcal{H}_{\nu}[\varphi_k]$.
\begin{theorem}
Let $m_0, m_{d}\geq m$, then
\begin{align}
\mathcal{H}_{\nu}[f] - Q^{F}[f] =
\mathcal{O}(\omega^{-m-\frac{3}{2}}), \quad \omega\rightarrow\infty.
\end{align}
\end{theorem}
\begin{proof}
Let $h(x) = f(x) - p(x)$ and thus $\mathcal{H}_{\nu}[f] - Q^{F}[f] =
\mathcal{H}_{\nu}[h]$. Using the expansion \eqref{eq:AsymCaseOne} we
obtain
\[
\mathcal{H}_{\nu}[f] - Q^{F}[f] \sim  - \sum_{k=1}^{\infty}
\frac{1}{(-\omega)^{k}} \left\{ \frac{\sigma_{k-1}[h](b)}{g{'}(b)}
J_{\nu+k}(\omega g(b)) - \frac{\sigma_{k-1}[h](a)}{g{'}(a)}
J_{\nu+k}(\omega g(a)) \right\}.
\]
Note that each $\sigma_{k}[h](x)$ is a linear combination of
$h(x),h{'}(x),\ldots,h^{(k)}(x)$. This together with the condition
$m_0,m_d \geq m$ implies
\[
\sigma_{k}[h](a) = \sigma_{k}[h](b) = 0, \quad k=0,\ldots,m-1.
\]
Hence the desired result follows.
\end{proof}

We reconsider the previous example using three different modified
Filon-type methods: $Q^{F}[f]$ with nodes $\{1,2\}$ and
multiplicities all one, $Q^{F}[f]$ with nodes
$\{1,\frac{4}{3},\frac{5}{3},2\}$ and multiplicities all one and
$Q^{F}[f]$ with nodes $\{1,\frac{4}{3},\frac{5}{3},2\}$ and
multiplicities all two. The modified moments
$\{\mu_{k,\nu}^{\omega}\}_{k=0}^{n-1}$ are evaluated by using
Theorem \ref{thm:ModMomRecRel} in the appendix. Numerical results
are shown in Figure \ref{fig:Filon}. As can be seen, the accuracy of
the proposed method can be improved greatly by adding more
derivatives interpolation at both endpoints or by adding more
interior nodes. Moreover, a desired accuracy level can be achieved
only with a small number of nodes and multiplicities.

\begin{figure}[ht]
\centering
\includegraphics[width=7.2cm,height=6cm]{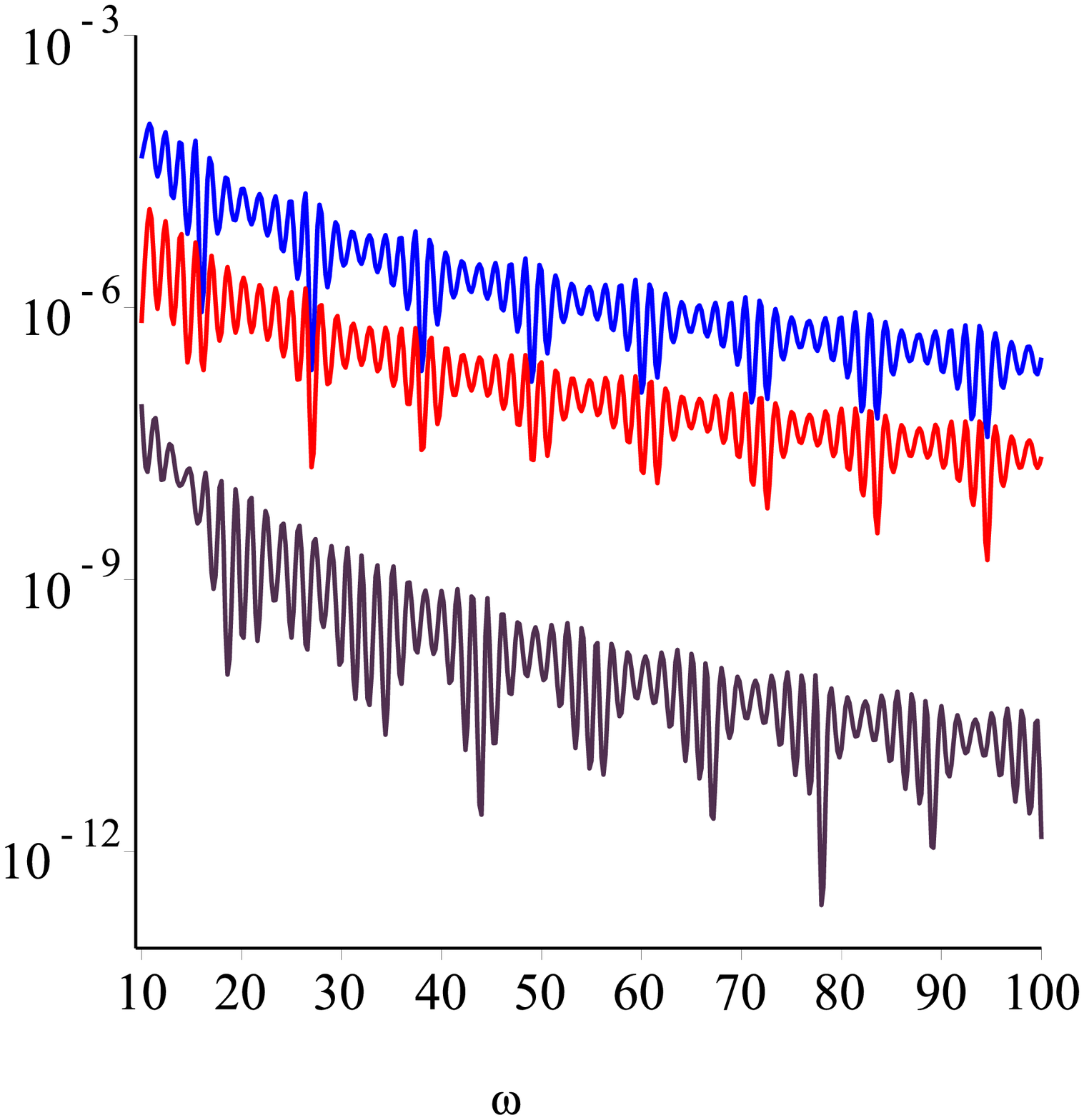}~~
\includegraphics[width=7.2cm,height=6cm]{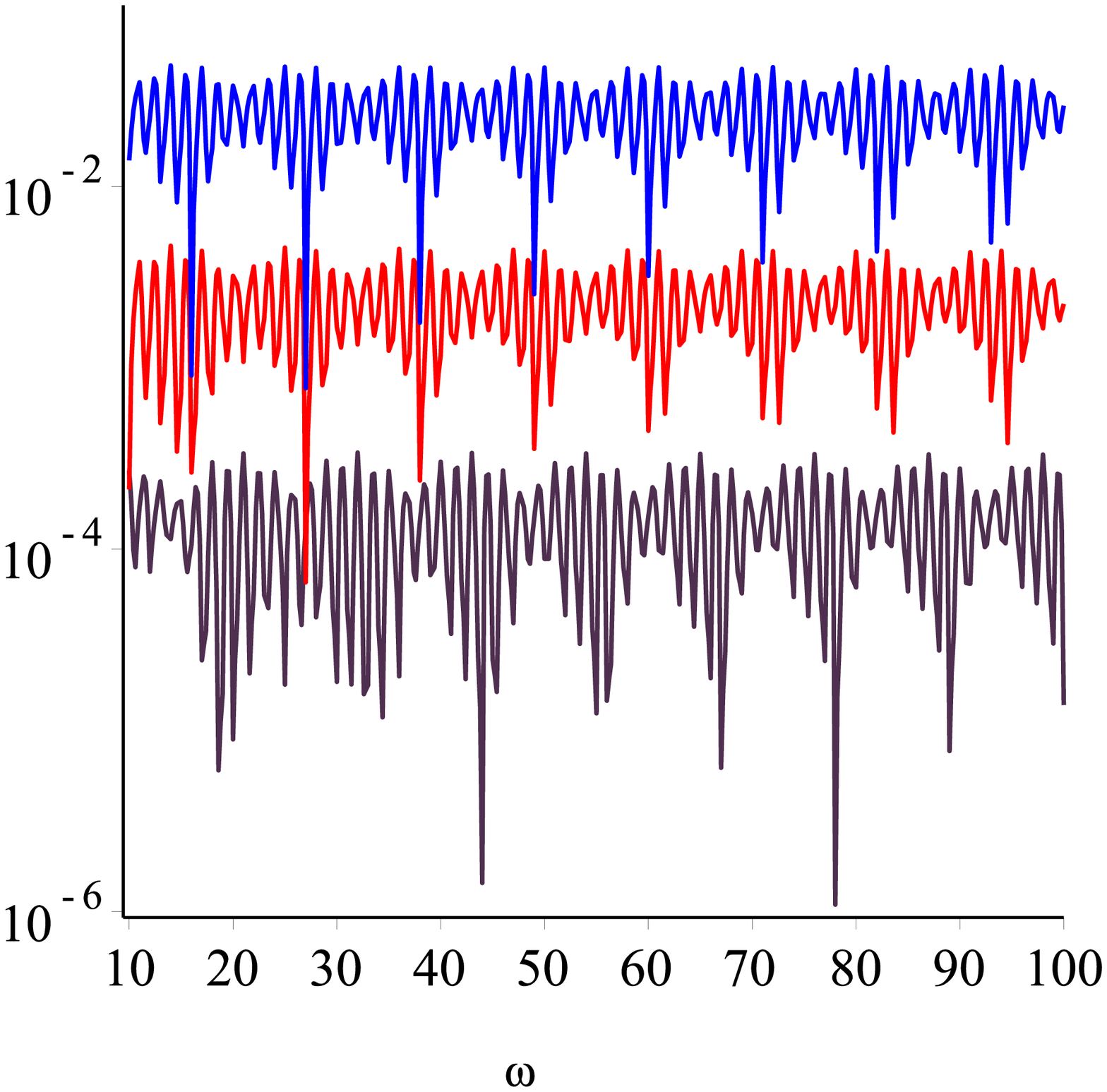}
\caption{The left panel shows the error of $Q^{F}[f]$ with nodes
$\{1,2\}$ and multiplicities all one (blue), $Q^{F}[f]$ with nodes
$\{1,\frac{4}{3},\frac{5}{3},2\}$ and multiplicities all one (red)
and $Q^{F}[f]$ with nodes $\{1,\frac{4}{3},\frac{5}{3},2\}$ and
multiplicities all two (violet). The right panel shows the error of
the first two methods scaled by $\omega^{\frac{5}{2}}$ and the error
of the last method scaled by $\omega^{\frac{7}{2}}$. }
\label{fig:Filon}
\end{figure}

\subsection{The case of zeros}\label{sec:zero}
In this subsection we consider the case that $g(x)$ has zeros in the
integration interval $[a,b]$ and we shall assume that $\Re(\nu)>-1$
to ensure the existence of the transform. Moreover, we restrict
ourselves to the case where the function $g(x)$ has a single zero on
the interval $[a,b]$, e.g., $g(\xi)=0$ where $\xi\in[a,b]$ and
$g(x)\neq0$ for $x\in[a,b]\backslash\{\xi\}$. If $g(x)$ has a finite
number of zeros on the interval $[a,b]$, we can divide the whole
interval into subintervals such that on each subinterval the
function $g(x)$ contains only a single zero.

Before proceeding, let us explain why the expansion in Theorem
\ref{thm:AsympCaseOne} fails when $g(x)$ has a zero in $[a,b]$.
Suppose that $\sigma_{k}[f](x) \in C^{\infty}[a,b]$ for some
$k\geq0$, it follows from \eqref{def:SigmaCase1} that
\begin{align}\label{eq:progression}
\sigma_{k+1}[f](x) &= g(x)^{\nu+k+1}
\frac{d}{dx}\left[\frac{\sigma_{k}[f](x)}{g(x)^{\nu+k+1} g{'}(x)}
\right] \nonumber \\
&= \frac{\sigma_{k}[f]{'}(x)}{g{'}(x)} - \frac{\sigma_{k}[f](x)
g{''}(x)}{g{'}(x)^2} - (\nu+k+1) \frac{\sigma_{k}[f](x)}{g(x)}
\nonumber
\\
&= \frac{d}{dx}\left[ \frac{\sigma_{k}[f](x)}{g{'}(x)} \right] -
(\nu+k+1) \frac{\sigma_{k}[f](x)}{g(x)}.
\end{align}
We can see clearly that a simple pole at $x=\xi$ is introduced in
the term $\sigma_{k}[f](x)/g(x)$ and thus integration by parts in
Theorem \ref{thm:AsympCaseOne} is no longer valid. To remedy this,
we may subtract the value $\sigma_{k}[f](\xi)$ from
$\sigma_{k}[f](x)$ before performing integration by parts and thus
the resulting $\sigma_{k+1}[f](x)$ has a removable singularity at
$x=\xi$.

\begin{theorem}\label{thm:AsympZero}
Let $f(x),g(x)\in C^{\infty}[a,b]$ and $g{'}(x)\neq0$ for
$x\in[a,b]$. Furthermore, assume that $g(\xi)=0$ for some
$\xi\in[a,b]$ and $g(x)\neq0$ for $x\in[a,b]\setminus\{\xi\}$. Then,
for $\omega\rightarrow\infty$, 
\begin{align}\label{eq:FullAsympZero}
\mathcal{H}_{\nu}[f] &\sim \sum_{k=0}^{\infty}
\frac{\widetilde{\sigma}_k[f](\xi)}{(-\omega)^k} M(\nu+k,\omega) -
\sum_{k=1}^{\infty} \frac{1}{(-\omega)^k} \left\{
\frac{\widetilde{\sigma}_{k-1}[f](b) -
\widetilde{\sigma}_{k-1}[f](\xi)}{g{'}(b)}
J_{\nu+k}(\omega g(b)) \right. \nonumber \\
&~~~~~ \left.  - \frac{\widetilde{\sigma}_{k-1}[f](a) -
\widetilde{\sigma}_{k-1}[f](\xi)}{g{'}(a)} J_{\nu+k}(\omega g(a))
\right\},
\end{align}
where $\widetilde{\sigma}_k[f](x)$ are defined by
\begin{align}\label{eq:SigmaZero}
\widetilde{\sigma}_0[f](x) = f(x), \quad \widetilde{\sigma}_k[f](x)
= g(x)^{\nu+k} \frac{d}{dx}\left[ \frac{
\widetilde{\sigma}_{k-1}[f](x) - \widetilde{\sigma}_{k-1}[f](\xi) }{
g(x)^{\nu+k} g{'}(x) } \right], \quad k\geq 1.
\end{align}
\end{theorem}
\begin{proof}
For each $m\geq0$, we first rewrite
$\mathcal{H}_{\nu+m}[\widetilde{\sigma}_m[f]]$ as
\begin{align}
\mathcal{H}_{\nu+m}[\widetilde{\sigma}_m[f]] &= \int_{a}^{b}
\widetilde{\sigma}_m[f](x)
J_{\nu+m}(\omega g(x)) dx \nonumber \\
&= \widetilde{\sigma}_m[f](\xi) \int_{a}^{b} J_{\nu+m}(\omega g(x))
dx + \int_{a}^{b} ( \widetilde{\sigma}_m[f](x) -
\widetilde{\sigma}_m[f](\xi) )
J_{\nu+m}(\omega g(x)) dx \nonumber \\
&= \widetilde{\sigma}_m[f](\xi) M(\nu+m,\omega) + \int_{a}^{b} (
\widetilde{\sigma}_m[f](x) - \widetilde{\sigma}_m[f](\xi) )
J_{\nu+m}(\omega g(x)) dx. \nonumber
\end{align}
Making use of integration by part to the last term, we find
\begin{align}
\mathcal{H}_{\nu+m}[\widetilde{\sigma}_m[f]]
&= \widetilde{\sigma}_{m}[f](\xi) M(\nu+m,\omega)  \nonumber \\
&~~~~~ + \frac{1}{\omega} \int_{a}^{b}
\frac{\widetilde{\sigma}_{m}[f](x) -
\widetilde{\sigma}_{m}[f](\xi)}{g{'}(x) g(x)^{\nu+m+1}}
\frac{d}{dx}\left[
g(x)^{\nu+m+1} J_{\nu+m+1}(\omega g(x)) \right] dx \nonumber \\
&= \widetilde{\sigma}_{m}[f](\xi) M(\nu+m,\omega) \nonumber \\
&~~~~~ + \frac{1}{\omega} \left\{ \frac{\widetilde{\sigma}_{m}[f](b)
- \widetilde{\sigma}_{m}[f](\xi)}{g{'}(b)} J_{\nu+m+1}(\omega g(b))
\right.
\nonumber \\
&~~~~~ \left. - \frac{\widetilde{\sigma}_{m}[f](a) -
\widetilde{\sigma}_{m}[f](\xi)}{g{'}(a)} J_{\nu+m+1}(\omega g(a))
\right\} - \frac{1}{\omega}
\mathcal{H}_{\nu+m+1}[\widetilde{\sigma}_{m+1}[f]]. \nonumber
\end{align}
Finally, the expansion \eqref{eq:FullAsympZero} can be obtained by
iterating the above process from $m=0$ to $m\rightarrow\infty$.

\end{proof}

The following corollary is an immediate consequence of Theorem
\ref{thm:AsympZero}, which has been proved in \cite{xiang2010fast}.
\begin{corollary}\label{cor:AsyZeroSim}
In the case $g(x)=x$ and $0=a<b$, then
\begin{align}
\mathcal{H}_{\nu}[f] &\sim \sum_{k=0}^{\infty}
\frac{\widetilde{\sigma}_{k}[f](0)}{(-\omega)^k} M(\nu+k,\omega) -
\sum_{k=1}^{\infty} \frac{\widetilde{\sigma}_{k-1}[f](b) -
\widetilde{\sigma}_{k-1}[f](0)}{(-\omega)^{k}} J_{\nu+k}(\omega b).
\nonumber
\end{align}
where $\widetilde{\sigma}_{k}[f](x)$ are defined as in
\eqref{eq:SigmaZero}.
\end{corollary}
\begin{proof}
Note that $\xi=0$, the desired result follows from
\eqref{eq:FullAsympZero}.
\end{proof}

We now turn to investigate the dependence of
$\widetilde{\sigma}_k[f](x)$ on $f(x)$ and its derivatives. When
$x\neq \xi$, it is easy to verify that each
$\widetilde{\sigma}_{k}[f](x)$ is a linear combination of
$f(x),f{'}(x),\ldots,f^{(k)}(x)$. When $x=\xi$, applying
L{'}H\^{o}spital's rule, we obtain
\begin{align}\label{eq:SigXi}
\widetilde{\sigma}_{1}[f](\xi) &= - \frac{\nu}{g{'}(\xi)} f{'}(\xi), \nonumber \\
\widetilde{\sigma}_{2}[f](\xi) &= - \frac{\nu+1}{( g{'}(\xi) )^3}
\left[ \frac{\nu-3}{2} g{''}(\xi) f{'}(\xi) -
\frac{\nu-1}{2} g{'}(\xi) f{''}(\xi) \right], \\
\widetilde{\sigma}_{3}[f](\xi) &= - \frac{\nu+2}{(g{'}(\xi))^5}
\left[ \left( \frac{8\nu - \nu^2}{6} g{'}(\xi) g{'''}(\xi) +
\frac{\nu^2-8\nu+3}{2}
(g{''}(\xi))^2 \right) f{'}(\xi) \right. \nonumber  \\
&~~~~~~~~~  - \frac{\nu^2-4\nu+1}{2} g{'}(\xi) g{'''}(\xi)
f{''}(\xi) \left. + \frac{\nu^2-2\nu}{6} (g{'}(\xi))^2 f{'''}(\xi)
\right], \nonumber
\end{align}
and it is not difficult to verify that each
$\widetilde{\sigma}_{k}[f](\xi)$ is a linear combination of
$f{'}(\xi),f{''}(\xi),\ldots,$ $f^{(k)}(\xi)$. In the particular
case where $g(x) = x$, $\xi=0$ and assume that $f(x)$ is analytic
inside a neighborhood of the origin, the values
$\widetilde{\sigma}_{k}[f](0)$ are a scalar multiple of $f^{(k)}(0)$
(see \cite[Theorem C.3]{wang2018hankel} or
\cite{wang2011asymptotic})
\begin{align}\label{eq:SigmaValue}
\widetilde{\sigma}_{k}[f](0) = \frac{ f^{(k)}(0) }{k!}
\prod_{\ell=0}^{k-1} \left(1 + 2\ell - \nu - k \right), \quad
k\geq1.
\end{align}

Truncating \eqref{eq:FullAsympZero} after the first $m$ terms,
yields an asymptotic method
\begin{align}\label{def:AsymMethodZero}
Q_m^{A}[f] &= \sum_{k=0}^{m-1}
\frac{\widetilde{\sigma}_k[f](\xi)}{(-\omega)^k} M(\nu+k,\omega) -
\sum_{k=1}^{m} \frac{1}{(-\omega)^k} \left\{
\frac{\widetilde{\sigma}_{k-1}[f](b) -
\widetilde{\sigma}_{k-1}[f](\xi)}{g{'}(b)}
J_{\nu+k}(\omega g(b)) \right. \nonumber \\
&~~~~~~~~~~~~~~~ \left.  - \frac{\widetilde{\sigma}_{k-1}[f](a) -
\widetilde{\sigma}_{k-1}[f](\xi)}{g{'}(a)} J_{\nu+k}(\omega g(a))
\right\}.
\end{align}
Moreover, from the above discussion that the method
\eqref{def:AsymMethodZero} is a linear combination of the values
$f(a),f{'}(a),\ldots,f^{(m-1)}(a)$, $f(\xi),f{'}(\xi),\ldots,$
$f^{(m-1)}(\xi)$ and $f(b),f{'}(b),\ldots,f^{(m-1)}(b)$.

Now, we are in a position to prove the error estimate for the
asymptotic method \eqref{def:AsymMethodZero}. Before stating the
result, it will be helpful to show the estimate of $M(\nu,\omega)$.
\begin{lemma}\label{lem:MuEstim}
Let $g(x)\in C^{\infty}[a,b]$ and $g{'}(x)\neq0$ for $x\in[a,b]$.
Furthermore, assume that $g(\xi)=0$ for some $\xi\in[a,b]$ and
$g(x)\neq0$ for $x\in[a,b]\setminus\{\xi\}$. Then, for
$\omega\rightarrow\infty$,
\begin{align}
M(\nu,\omega) = \mathcal{O}(\omega^{-1}).
\end{align}
\end{lemma}
\begin{proof}
Since $g{'}(x)\neq0$ for $x\in[a,b]$ this implies that $g(x)$ is a
strictly monotonic function on the interval $[a,b]$. Without loss of
generality, we may assume that $g(x)$ is a monotonically increasing
function, i.e., $g(a)\leq 0 \leq g(b)$.

Split the integral representation of $M(\nu,\omega)$ into integrals
over $[a,\xi]$ and $[\xi,b]$, we have
\begin{align}
M(\nu,\omega) &= \int_{a}^{\xi} J_{\nu}(\omega g(x))dx +
\int_{\xi}^{b} J_{\nu}(\omega g(x))dx. \nonumber \\
&= \int_{0}^{-g(a)} \frac{1}{g{'}(g^{-1}(-t))} J_{\nu}(-\omega t) dt
+ \int_{0}^{g(b)} \frac{1}{g{'}(g^{-1}(t))} J_{\nu}(\omega t)dt,
\nonumber
\end{align}
where we have made the change of variable $t=-g(x)$ in the first
integral and $t=g(x)$ in the second. By making use of the identity
$J_{\nu}(-z) = e^{\nu\pi i}J_{\nu}(z)$ \cite[Eqn.
10.11.1]{olver2010nist}, we can rewrite $M(\nu,\omega)$ as
\begin{align}\label{eq:MuEstim}
M(\nu,\omega) &= e^{\nu\pi i} \int_{0}^{-g(a)}
\frac{1}{g{'}(g^{-1}(-t))} J_{\nu}(\omega t) dt + \int_{0}^{g(b)}
\frac{1}{g{'}(g^{-1}(t))} J_{\nu}(\omega t)dt.
\end{align}
Using Corollary \ref{cor:AsyZeroSim} and keeping only the leading
terms, we have
\begin{align}\label{eq:EstimMom}
M(\nu,\omega) &\sim \frac{1}{g{'}(\xi)} \left[ e^{\nu\pi i}
\int_{0}^{-g(a)} J_{\nu}(\omega t) dt
+ \int_{0}^{g(b)} J_{\nu}(\omega t) dt \right] \nonumber \\
&~~~~~ + \frac{e^{\nu\pi i}}{\omega} \left[ \frac{1}{g{'}(a)} -
\frac{1}{g{'}(\xi)} \right] J_{\nu+1}(-\omega g(a)) +
\frac{1}{\omega} \left[ \frac{1}{g{'}(b)} - \frac{1}{g{'}(\xi)}
\right] J_{\nu+1}(\omega g(b)). \nonumber
\end{align}
To estimate $M(\nu,\omega)$, it suffices to derive estimates of
those terms on the right-hand side. For the first term, from Lemma
\ref{lem:EstimMom} in the appendix we see that both integrals inside
the bracket behave like $\mathcal{O}(\omega^{-1})$ for fixed $\nu$
and $\Re(\nu)>-1$. For the second and third terms, it is easy to see
that both terms behave like $\mathcal{O}(\omega^{-\frac{3}{2}})$.
Hence, the desired result follows.
\end{proof}

Owing to Lemma \ref{lem:MuEstim}, we have the following estimate.
\begin{theorem}\label{thm:AsymErrorZero}
Under the same assumptions as in Theorem \ref{thm:AsympZero}, it
follows that, for $\omega\rightarrow\infty$,
\begin{align}
\mathcal{H}_{\nu}[f] - Q_m^{A}[f] = \left\{\begin{array}{cc}
                                           {\displaystyle \mathcal{O}(\omega^{-m-1}) },  & \mbox{if $\widetilde{\sigma}_{m}[f](\xi)\neq0 $}, \\ [10pt]
                                           {\displaystyle \mathcal{O}(\omega^{-m-\frac{3}{2}}) },  & \mbox{if $\widetilde{\sigma}_{m}[f](\xi)=0
                                           $},
                                        \end{array}
                                        \right.
\end{align}
where $Q_m^{A}[f]$ is defined as in \eqref{def:AsymMethodZero}.
\end{theorem}
\begin{proof}
Combining \eqref{eq:FullAsympZero} and \eqref{def:AsymMethodZero},
we see that
\begin{align}\label{eq:ErrorAsympZero}
\mathcal{H}_{\nu}[f] - Q_m^{A}[f] &\sim \sum_{k=m}^{\infty}
\frac{\widetilde{\sigma}_k[f](\xi)}{(-\omega)^k}
M(\nu+k,\omega) \nonumber \\
&~~~~~ - \sum_{k=m+1}^{\infty} \frac{1}{(-\omega)^k} \left\{
\frac{\widetilde{\sigma}_{k-1}[f](b) -
\widetilde{\sigma}_{k-1}[f](\xi)}{g{'}(b)}
J_{\nu+k}(\omega g(b)) \right.  \\
&~~~~~ \left. - \frac{\widetilde{\sigma}_{k-1}[f](a) -
\widetilde{\sigma}_{k-1}[f](\xi)}{g{'}(a)} J_{\nu+k}(\omega g(a))
\right\}. \nonumber
\end{align}
It is clear that the behavior of the left-hand side is determined by
comparing the behavior of those two sums on the right-hand side.
Thanks to Lemma \ref{lem:MuEstim}, we can deduce that the first sum
on the right-hand side of \eqref{eq:ErrorAsympZero} behaves like
$\mathcal{O}(\omega^{-m-1})$ if $\widetilde{\sigma}_m[f](\xi)\neq0$
and $\mathcal{O}(\omega^{-m-2})$ if
$\widetilde{\sigma}_m[f](\xi)=0$. For the second sum, we easily see
that it behaves like $\mathcal{O}(\omega^{-m-\frac{3}{2}})$.
Combining these estimates gives the desired result.
\end{proof}

As a simple example, we consider $f(x)=\sin x$, $g(x)=x$, $\nu=2$
and $[a,b]=[0,1]$. Clearly, $g(x)$ has a unique zero at $\xi=0$.
Moreover, from \eqref{eq:SigmaValue} we can check easily that
$\sigma_1[f](\xi)=-2$ and $\sigma_{k}[f](\xi)=0$ for $k=2,3$. Thus,
we can expect that the decay rate of the error of $Q_m^{A}[f]$ is
$\mathcal{O}(\omega^{-m-1})$ for $m=1$ and is
$\mathcal{O}(\omega^{-m-\frac{3}{2}})$ for $m=2,3$. In our
computations, the moments $M(\nu+k,\omega)$ are evaluated by
\eqref{eq:MomLom} in the appendix directly. Numerical results are
illustrated in Figure \ref{fig:AsyZero} and we can see that they are
in good agreement with our analysis.

\begin{figure}[ht]
\centering
\includegraphics[width=7.2cm,height=6cm]{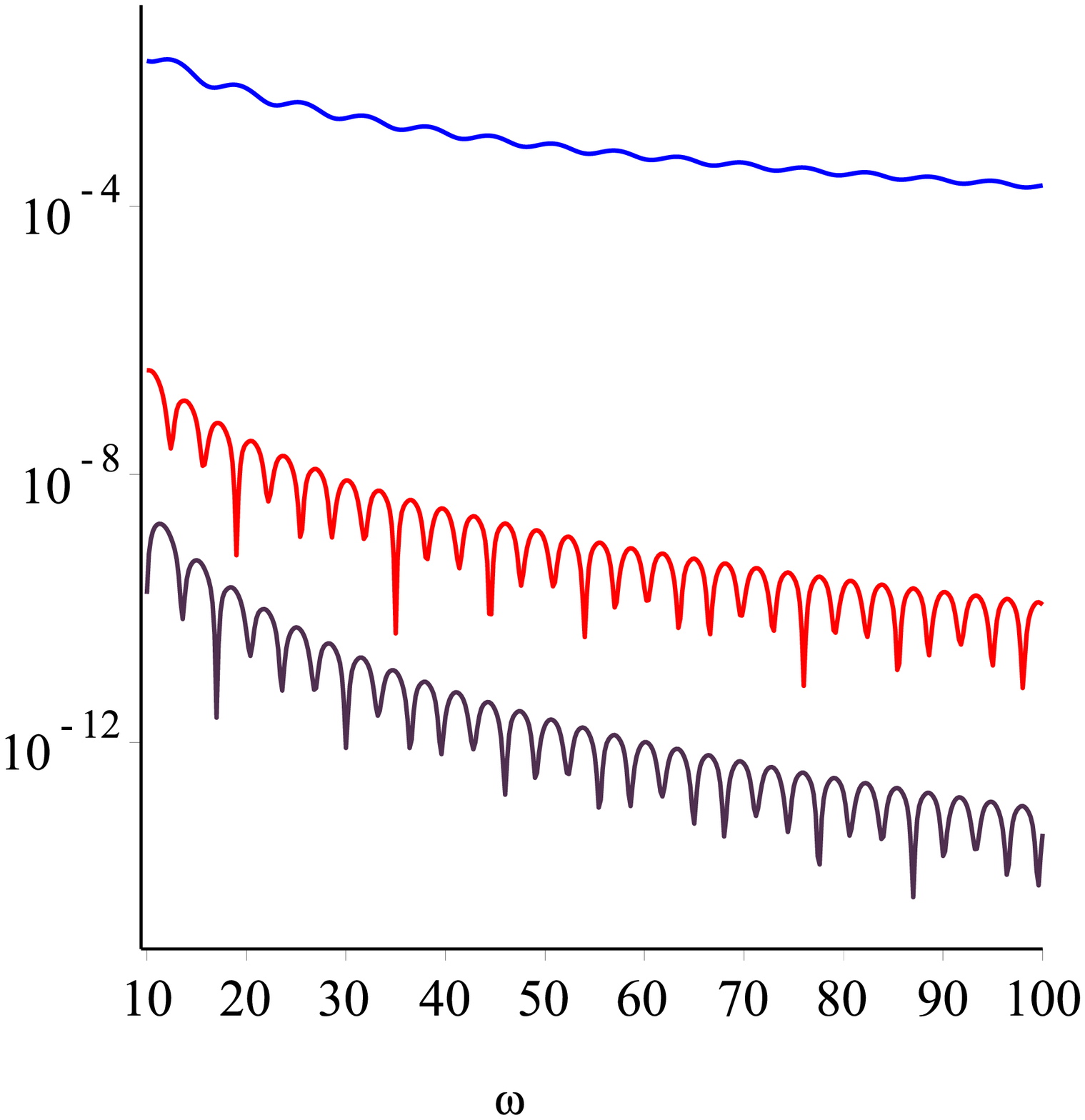}~
\includegraphics[width=7.2cm,height=6cm]{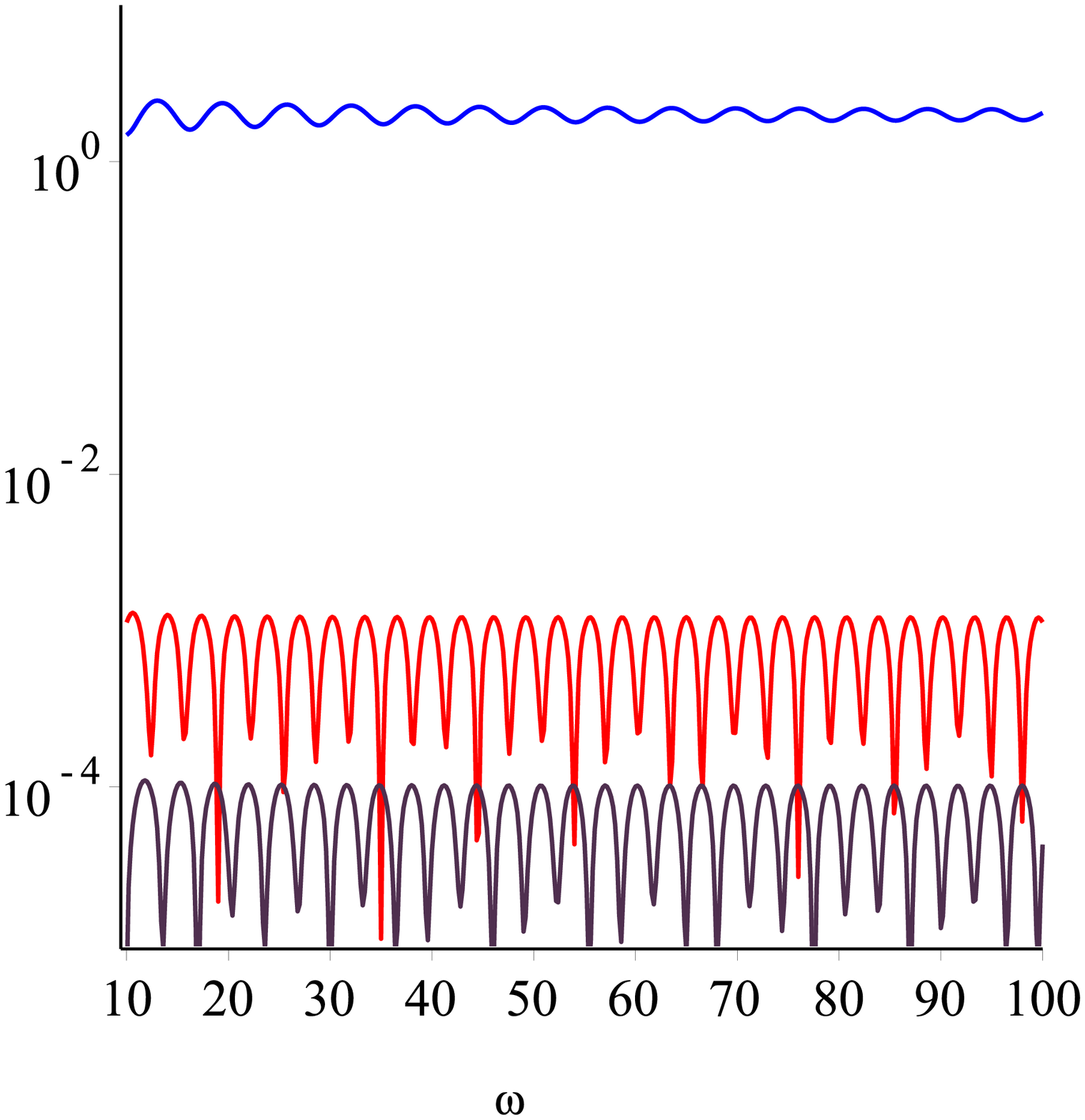}
\caption{The error of $Q_m^{A}[f]$ (left) and the error of
$Q_{1}^{A}[f]$ scaled by $\omega^2$ and the error of $Q_{m}^{A}[f]$
scaled by $\omega^{m+\frac{3}{2}}$ for $m=2,3$ (right). 
The blue, red and violet curves correspond to $m=1,2,3$,
respectively.} \label{fig:AsyZero}
\end{figure}

\begin{remark}
When using the asymptotic method \eqref{def:AsymMethodZero}, it is
necessary to evaluate the moments $M(\nu+k,\omega)$. As will be
shown below, modified Filon-type methods in Theorem
\ref{thm:FilonZero} can be used to evaluate them accurately; see
Figure \ref{fig:FilonZero} for an illustration.
\end{remark}

We now consider to construct a modified Filon-type method. Note that
the zero of $g(x)$ is a critical point of the transform, so we need
to impose interpolation conditions at this point. Let $x_\ell = \xi$
for some $\ell\in \{0,\ldots,d\}$ and let $p(x)=\sum_{k=0}^{n-1} c_k
\varphi_k(x) \in \mathcal{E}$ be the unique function that satisfies
the interpolation conditions \eqref{eq:HermiteBirkhoff} and let
$Q^{F}[f] = \mathcal{H}_{\nu}[p]$.
\begin{theorem}\label{thm:FilonZero}
Let $m_0,m_{\ell},m_d \geq m$, then
\begin{align}
\mathcal{H}_{\nu}[f]-Q^{F}[f] = \left\{\begin{array}{cc}
                {\displaystyle \mathcal{O}(\omega^{-m-1}) },  & \mbox{if $\widetilde{\sigma}_{m}[f-p](\xi)\neq 0$}, \\ [10pt]
                {\displaystyle \mathcal{O}(\omega^{-m-\frac{3}{2}}) },  & \mbox{if $\widetilde{\sigma}_{m}[f-p](\xi)=0$}.
                  \end{array}
                  \right.
\end{align}
\end{theorem}
\begin{proof}
Let $h(x) = f(x) - p(x) $ and we clearly have
$\mathcal{H}_{\nu}[f]-Q^{F}[f] = \mathcal{H}_{\nu}[h]$.
Additionally, from the interpolation conditions and the dependence
of $\widetilde{\sigma}_k[f](x)$ on $f(x)$ and its derivatives, we
obtain
\begin{align*}
\widetilde{\sigma}_{k}[h](a) = \widetilde{\sigma}_{k}[h](\xi) =
\widetilde{\sigma}_{k}[h](b) = 0, \quad k=0,\ldots,m-1.
\end{align*}
The result follows from Theorem \ref{thm:AsymErrorZero}.
\end{proof}

\begin{figure}[ht]
\centering
\includegraphics[width=7.2cm,height=6cm]{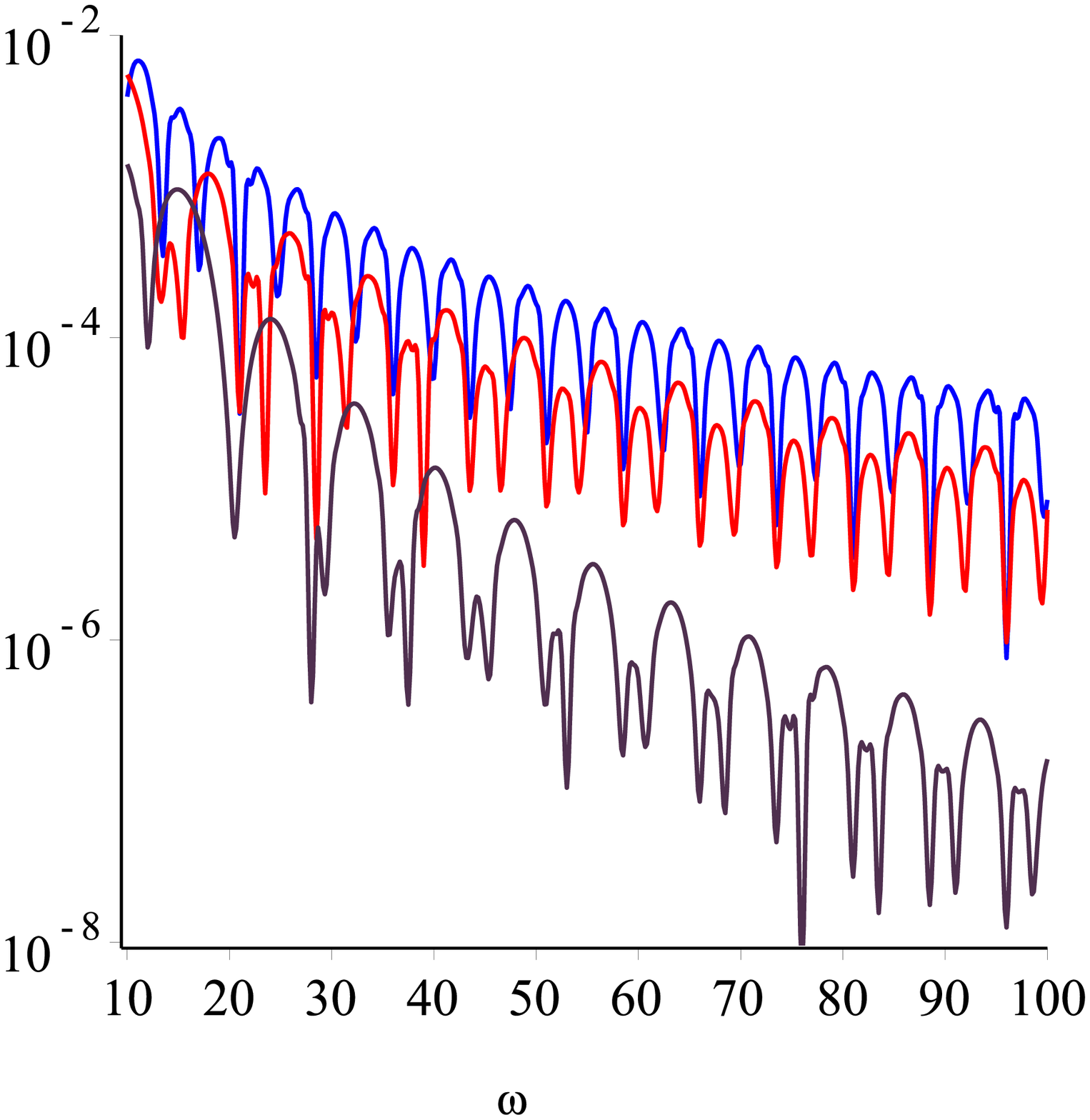}~
\includegraphics[width=7.2cm,height=6cm]{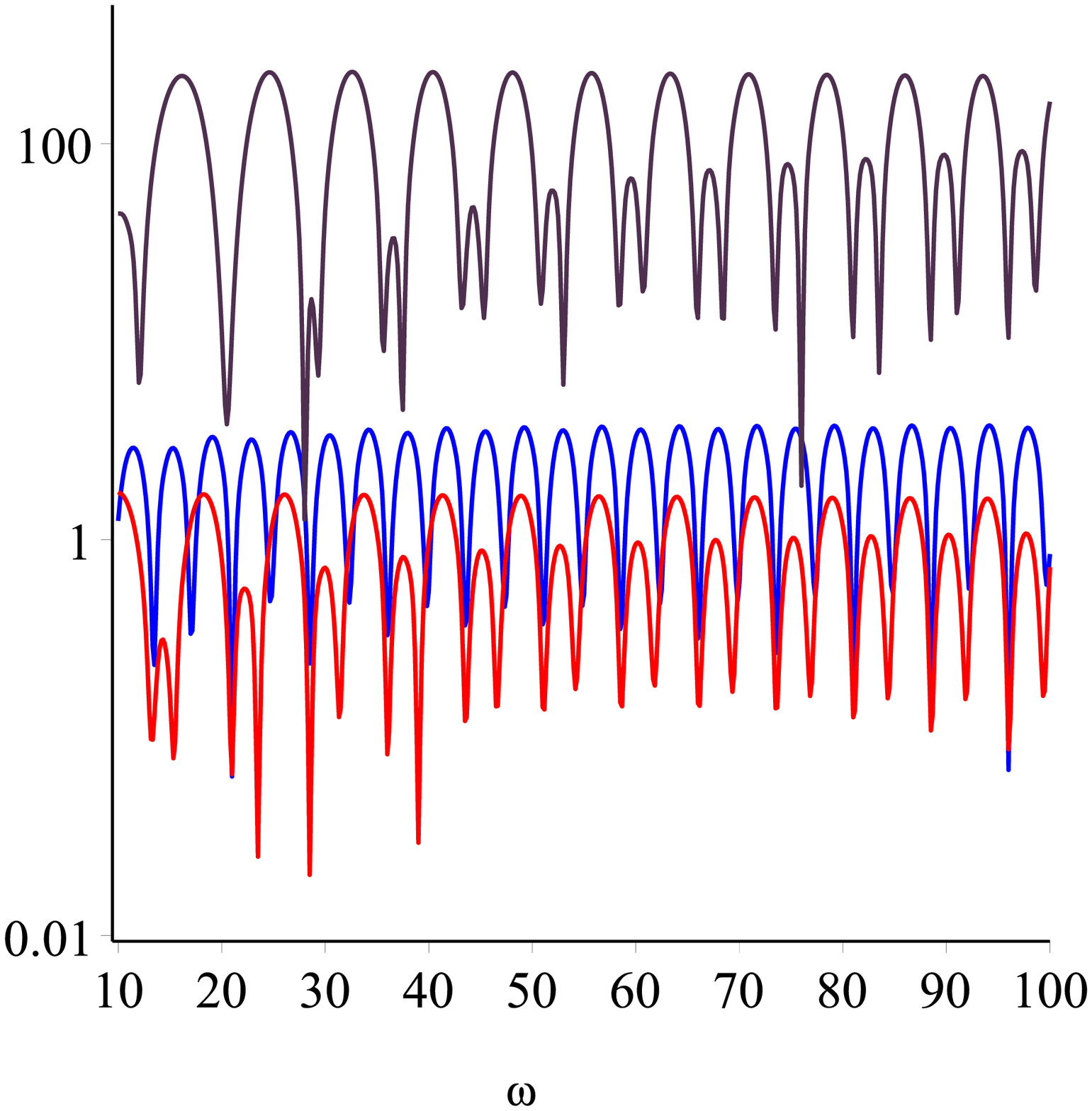}
\caption{The left panel shows the error of $Q^{F}[f]$ with nodes
$\{0,1\}$ and multiplicities all one (blue), $Q^{F}[f]$ with the
nodes $\{0,\frac{1}{3},\frac{2}{3},1\}$ and multiplicities all one
(red) and $Q^{F}[f]$ with nodes $\{0,\frac{1}{3},\frac{2}{3},1\}$
and multiplicities $\{3,1,1,3\}$ (violet). The right panel shows the
error of the first two methods scaled by $\omega^{\frac{5}{2}}$ and
the last one scaled by $\omega^{\frac{9}{2}}$.}
\label{fig:FilonZero}
\end{figure}

As an example, we consider the evaluation of the following moment
\begin{align}
M(0,\omega) = \int_{0}^{1} J_0(\omega \sin x) dx.
\end{align}
This moment is a finite Hankel transform with $f(x)=1$, $g(x)=\sin
x$, $\nu=0$ and $\xi=0$. We compare the accuracy of three modified
Filon-type methods: $Q^{F}[f]$ with nodes $\{0,1\}$ and
multiplicities all one, $Q^{F}[f]$ with the nodes
$\{0,\frac{1}{3},\frac{2}{3},1\}$ and multiplicities all one and
$Q^{F}[f]$ with the nodes $\{0,\frac{1}{3},\frac{2}{3},1\}$ and
multiplicities $\{3,1,1,3\}$. For the first two methods, they
correspond to $m=1$ and it is easy to check from \eqref{eq:SigXi}
that $\widetilde{\sigma}_{1}[f-p](0)=0$. Thus, the decay rate of the
error of both methods is $\mathcal{O}(\omega^{-\frac{5}{2}})$. For
the last method, we can see that $m=3$ and we can check from
\eqref{eq:SigXi} that
$\widetilde{\sigma}_{3}[f-p](0)=-(f{''}(0)-p{''}(0))=0$. Therefore,
the decay rate of its error is $\mathcal{O}(\omega^{-\frac{9}{2}})$.
In our implementation, the modified moments
$\{\mu_{k,\nu}^{\omega}\}_{k=0}^{n-1}$ are evaluated by Theorem
\ref{eq:ModMomRecRel} in the appendix. Numerical results are
presented in Figure \ref{fig:FilonZero}. Clearly, we can see that
the decay rates of these three methods are consistent with the
expected rates.

\section{The case of stationary points}\label{sec:Stationary}
When $g{'}(\zeta)=0$ for some $\zeta\in[a,b]$, we call $x=\zeta$ a
stationary point. Further, we call $x=\zeta$ a stationary point of
order $r$ if $g^{(j)}(\zeta) = 0$, $j=1,\ldots,r$, where $r\geq1$
and $g^{(r+1)}(\zeta) \neq 0$.

In this section, we restrict our attention to the case that $g(x)$
has only one stationary point of order $r$ at $x=\zeta\in[a,b]$,
i.e., $g{'}(x)\neq0$ for $x\in[a,b]\backslash \{\zeta\}$. We also
assume that $g(x)\neq0$ for $x\in[a,b]\backslash\{\zeta\}$. The
stationary point can be classified into the following two types:
\begin{enumerate}
\item Type I: $g(\zeta)\neq0$.

\item Type II: $g(\zeta)=0$.
\end{enumerate}
Here we shall restrict our analysis to the latter case, i.e.,
$x=\zeta$ is a stationary point of type II, since a similar analysis
can be performed for the former case. Moreover, we point out that if
$g(x)$ has a zero at $x=\overline{\zeta}\in[a,b]$ and
$\overline{\zeta} \neq \zeta$, then we can divide the interval
$[a,b]$ into two subintervals such that one contains the zero
$x=\overline{\zeta}$ and the other contains the stationary point
$x=\zeta$. To ensure the existence of the transform, we shall assume
throughout this section that $\Re(\nu)>-\frac{1}{r+1}$.

\subsection{Asymptotic expansions}
In order to gain insight into the mechanism of performing
integration by parts when $g(x)$ has a stationary point, we return
to Theorem \ref{thm:AsympCaseOne} again. From \eqref{eq:progression}
we can see that a pole at $x=\zeta$ is introduced in the term inside
the bracket $\sigma_{k}[f](x)/g{'}(x)$ and the term
$\sigma_{k}[f](x)/g(x)$. The existence of pole implies that
\eqref{def:SigmaCase1} is still no longer valid. However, this
problem can be circumvented by subtracting the first several terms
of $\sigma_k[f](x)$ about $x=\zeta$ before we perform integration by
parts. Note that both $\sigma_{k}[f](x)/g{'}(x)$ and
$\sigma_{k}[f](x)/g(x)$ have a pole at $x=\zeta$ and the pole is of
order $r$ for the former and is of order $r+1$ for the latter. It is
sufficient to subtract the first $r+1$ terms of the Taylor expansion
of $\sigma_k[f](x)$ about $x=\zeta$ to ensure that the resulting
$\sigma_{k+1}[f](x)$ has a removable singularity.

For simplicity in exposition, let us define
\begin{align}\label{eq:TaylorTrunc}
T_r[f](x,y) = \sum_{j=0}^{r} \frac{f^{(j)}(y)}{j!} (x-y)^j.
\end{align}
We start from the case that $x=\zeta$ is an interior stationary
point, i.e., $\zeta\in(a,b)$.
\begin{theorem}\label{thm:AsySta}
Assume that $f(x),g(x)\in C^{\infty}[a,b]$ and $g(x)$ has a
stationary point of type II and of order $r$ at $x=\zeta\in(a,b)$.
Furthermore, we assume that $g(x),g{'}(x)\neq0$ for $x\in
[a,b]\setminus\{\zeta\}$. Then, for $\omega\rightarrow\infty$,
\begin{align}\label{eq:FullAsySta}
\mathcal{H}_{\nu}[f] &\sim \sum_{k=0}^{\infty} \frac{1}{(-\omega)^k}
\sum_{j=0}^{r} \frac{ \widehat{\sigma}_{k}[f]^{(j)}(\zeta) }{j!}
M_j(\zeta,\nu+k,\omega) \nonumber
\\
&~~~~~ - \sum_{k=1}^{\infty} \frac{1}{(-\omega)^k} \left\{
\frac{\widehat{\sigma}_{k-1}[f](b) -
T_{r}[\widehat{\sigma}_{k-1}](b,\zeta)}{g{'}(b)} J_{\nu+k}(\omega
g(b)) \right. \\
&~~~~~ \left. - \frac{\widehat{\sigma}_{k-1}[f](a) -
T_{r}[\widehat{\sigma}_{k-1}](a,\zeta)}{g{'}(a)} J_{\nu+k}(\omega
g(a)) \right\}, \nonumber
\end{align}
where
\begin{align}\label{eq:FulSigSta}
\widehat{\sigma}_{0}[f](x) = f(x),~~ \widehat{\sigma}_{k}[f](x) =
g(x)^{\nu+k} \frac{d}{dx} \left[ \frac{\displaystyle
\widehat{\sigma}_{k-1}[f](x) -
T_{r}[\widehat{\sigma}_{k-1}](x,\zeta) }{ g(x)^{\nu+k} g{'}(x) }
\right],~ k\geq1.
\end{align}
\end{theorem}
\begin{proof}
The proof is similar to that of Theorem \ref{thm:AsympZero}. We omit
the details.
\end{proof}

When the stationary point $x=\zeta$ is at an endpoint, the result in
Theorem \ref{thm:AsySta} should be modified slightly. Without loss
of generality, we assume that the stationary point is located at the
left endpoint, i.e., $\zeta=a$.
\begin{theorem}\label{thm:EndAsySta}
Assume that $f(x),g(x)\in C^{\infty}[a,b]$ and $g(x)$ has a
stationary point of type II and of order $r$ at $x=a$. Furthermore,
we assume that $g(x),g{'}(x)\neq0$ for $x\in (a,b]$. Then, for
$\omega\rightarrow\infty$,
\begin{align}\label{eq:EndAsySt2}
\mathcal{H}_{\nu}[f] &\sim \sum_{k=0}^{\infty} \frac{1}{(-\omega)^k}
\sum_{j=0}^{r} \frac{ \widehat{\sigma}_{k}[f]^{(j)}(a) }{j!}
M_j(a,\nu+k,\omega) \nonumber
\\
&~~~~~ - \sum_{k=1}^{\infty} \frac{1}{(-\omega)^k}
\left[\frac{\widehat{\sigma}_{k-1}[f](b) -
T_{r}[\widehat{\sigma}_{k-1}](b,a)}{g{'}(b)} \right]
J_{\nu+k}(\omega g(b)),
\end{align}
where $\widehat{\sigma}_k[f](x)$ are defined as in
\eqref{eq:FulSigSta}.
\end{theorem}
\begin{proof}
Note that
\begin{align}
\lim_{x\rightarrow a} \frac{ \widehat{\sigma}_{k-1}[f](x) -
T_{r}[\widehat{\sigma}_{k-1}](x,a) }{g{'}(x)} =0. \nonumber
\end{align}
The expansion \eqref{eq:EndAsySt2} follow from Theorem
\ref{thm:AsySta}.
\end{proof}

As a consequence of Theorem \ref{thm:EndAsySta}, we have the
following result.
\begin{corollary}\label{cor:AsyStaSim}
In the simplest case $g(x)=x^{r+1}$ and $0=a<b$, then
\begin{align}\label{eq:AsymStExam}
\mathcal{H}_{\nu}[f] &\sim \sum_{k=0}^{\infty} \frac{1}{(-\omega)^k}
\sum_{j=0}^{r} \frac{ \widehat{\sigma}_{k}[f]^{(j)}(0)}{j!}
M_j(0,\nu+k,\omega) \nonumber
\\
&~~~~~ - \sum_{k=1}^{\infty} \frac{1}{(-\omega)^k} \bigg[
\widehat{\sigma}_{k-1}[f](b) - T_{r}[\widehat{\sigma}_{k-1}](b,0)
\bigg] \frac{J_{\nu+k}(\omega b^{r+1})}{(r+1)b^{r}}.
\end{align}
where $\widehat{\sigma}_k[f](x)$ are defined as in
\eqref{eq:FulSigSta}.
\end{corollary}
\begin{proof}
It follows immediately from Theorem \ref{thm:EndAsySta}.
\end{proof}

We now clarify the dependence of $\widehat{\sigma}_k[f](x)$ on
$f(x)$ and its derivatives. When $x\neq\zeta$, one can easily verify
that each $\widehat{\sigma}_k[f](x)$ is a linear combination of
$f(x), f{'}(x),\ldots,f^{(k)}(x)$. In the case $x=\zeta$ and assume
that $f$ and $g$ are analytic functions in a neighborhood of the
interval $[a,b]$, from Theorem \ref{thm:DepRelSigma} in the appendix
we know that each $\widehat{\sigma}_k[f]^{(j)}(\zeta)$ is a linear
combination of the values $f^{(r+1+j)}(\zeta),f^{(r+2+j)}(\zeta),$
$\ldots,f^{(k(r+1)+j)}(\zeta)$. 

As before, the asymptotic methods can be defined by truncating the
expansion in Theorem \ref{thm:AsySta} when $\zeta\in(a,b)$ and by
truncating the expansion in Theorem \ref{thm:EndAsySta} when
$\zeta=a$. For simplicity of presentation, we only consider the
latter case. The asymptotic method is defined by
\begin{align}\label{def:EndAsmMethII}
Q_{m}^{A}[f] &= \sum_{k=0}^{m-1} \frac{1}{(-\omega)^k}
\sum_{j=0}^{r} \frac{ \widehat{\sigma}_{k}[f]^{(j)}(a) }{j!}
M_j(a,\nu+k,\omega) \nonumber
\\
&~~~~~ - \sum_{k=1}^{m} \frac{1}{(-\omega)^k} \bigg[
\frac{\widehat{\sigma}_{k-1}[f](b) -
T_{r}[\widehat{\sigma}_{k-1}](b,a)}{g{'}(b)} \bigg] J_{\nu+k}(\omega
g(b)).
\end{align}
Moreover, the above discussion shows that the method $Q_m^{A}[f]$ is
a linear combination of the values
$f(b),f{'}(b),\ldots,f^{(m-1)}(b)$ and
$f(a),f{'}(a),\ldots,f^{(m(r+1)-1)}(a)$.

\subsection{Error estimate}
To derive error estimates of the asymptotic method $Q_m^{A}[f]$
defined in \eqref{def:EndAsmMethII}, we need to show the estimate of
$M_j(\zeta,\nu,\omega)$. 
\begin{lemma}\label{lem:EstMomSta}
Assume that $g(x)\in C^{\infty}[a,b]$ and $g(x)$ has a stationary
point of type II and of order $r$ at $x=\zeta\in[a,b]$. Furthermore,
assume that $g(x),g{'}(x)\neq0$ for $x\in [a,b]\setminus\{\zeta\}$.
Then, for $0\leq j \leq r$ and $\omega\rightarrow\infty$,
\begin{align}
M_j(\zeta,\nu,\omega) = \mathcal{O}(\omega^{-\frac{j+1}{r+1}}).
\end{align}
\end{lemma}
\begin{proof}
From the assumption $g{'}(x)\neq0$ for
$x\in[a,b]\setminus\{\zeta\}$, we see that $g(x)$ is strictly
monotonically on each subinterval $[a,\zeta]$ and $[\zeta,b]$. For
simplicity, we only consider the case where $g(x)$ is strictly
decreasing on $[a,\zeta]$ and is strictly increasing on $[\zeta,b]$,
i.e., $g(a)\geq0$ and $g(b)\geq0$. Splitting the integral
representation of $M_j(\zeta,\nu,\omega)$ at $x=\zeta$ gives
\begin{align}\label{eq:MomStat}
M_j(\zeta,\nu,\omega) &= \int_{a}^{\zeta} (x-\zeta)^j J_{\nu}(\omega
g(x)) dx +
\int_{\zeta}^{b} (x-\zeta)^j J_{\nu}(\omega g(x)) dx \nonumber \\
&= \int_{\sqrt[r+1]{g(a)}}^{0} y^j h(y) J_{\nu}(\omega y^{r+1}) dy +
\int_{0}^{\sqrt[r+1]{g(b)}} y^j h(y) J_{\nu}(\omega y^{r+1}) dy,
\end{align}
where in the last step we have made a change of variable
$y^{r+1}=g(x)$ and
\[
h(y) = (r+1)\frac{(x-\zeta)^j}{g{'}(x)} y^{r-j}.
\]
Note that the assumption on $g(x)$ implies that $h(y)$ is a
$C^{\infty}$ function and $h(0)\neq0$. We now consider the
asymptotic behavior of the second integral on the right-hand side of
\eqref{eq:MomStat}. 
Let us define $\widehat{h}(y) = y^j h(y)$ and $\widehat{b} =
\sqrt[r+1]{g(b)}$. By invoking Corollary \ref{cor:AsyStaSim} and
keeping only the first term of the expansion, we have
\begin{align}\label{eq:GenMomCanEst}
\int_{0}^{\sqrt[r+1]{g(b)}} y^j h(y) J_{\nu}(\omega y^{r+1}) dy &=
\int_{0}^{\widehat{b}} \widehat{h}(y) J_{\nu}(\omega y^{r+1}) dy
\nonumber \\
&\sim \sum_{\ell=0}^{r} \frac{ \widehat{h}^{(\ell)}(0)}{\ell!}
\int_{0}^{\widehat{b}} y^{\ell} J_{\nu}(\omega y^{r+1}) dy \nonumber
\\
&~~~~~ + \frac{1}{\omega} \bigg[ \widehat{h}(\widehat{b}) -
T_{r}[\widehat{h}](\widehat{b},0) \bigg] \frac{J_{\nu+1}(\omega
\widehat{b}^{r+1})}{(r+1) \widehat{b}^r} \nonumber \\
&= \sum_{\ell=j}^{r} \frac{ \widehat{h}^{(\ell)}(0)}{\ell!}
\int_{0}^{\widehat{b}} y^{\ell} J_{\nu}(\omega y^{r+1}) dy \nonumber
\\
&~~~~~ + \frac{1}{\omega} \bigg[ \widehat{h}(\widehat{b}) -
T_{r}[\widehat{h}](\widehat{b},0) \bigg] \frac{J_{\nu+1}(\omega
\widehat{b}^{r+1})}{(r+1) \widehat{b}^r},
\end{align}
where we have used the fact $\widehat{h}^{(\ell)}(0)=0$ for $0\leq
\ell\leq j-1$ in the last step. For the integrals on the right-hand
side of \eqref{eq:GenMomCanEst}, by setting $t = y^{r+1}$, we have
\[
\int_{0}^{\widehat{b}} y^{\ell} J_{\nu}(\omega y^{r+1}) dy =
\frac{1}{r+1} \int_{0}^{\widehat{b}^{r+1}} t^{\frac{\ell+1}{r+1}-1}
J_{\nu}(\omega t) dt.
\]
From Lemma \ref{lem:EstimMom} in the appendix we know that the
integral on the right-hand side behaves like
$\mathcal{O}(\omega^{-\frac{\ell+1}{r+1}})$ for fixed $\nu$ and
$\Re(\nu)>-\frac{1}{r+1}$. Thus, we can deduce that the first term
on the right-hand side of \eqref{eq:GenMomCanEst} behaves like
$\mathcal{O}(\omega^{-\frac{j+1}{r+1}})$. Moreover, it is easy to
see that the second term on the right-hand side of
\eqref{eq:GenMomCanEst} behaves like
$\mathcal{O}(\omega^{-\frac{3}{2}})$. We can conclude that the
second integral on the right-hand side of \eqref{eq:MomStat} behaves
like $\mathcal{O}(\omega^{-\frac{j+1}{r+1}})$.

Finally, using a similar argument, we can show that the first
integral on the right-hand side of \eqref{eq:MomStat} also behaves
like $\mathcal{O}(\omega^{-\frac{j+1}{r+1}})$. This completes the
proof.
\end{proof}

\begin{theorem}\label{thm:AsyErrorSta}
Under the same assumptions as in Theorem \ref{thm:EndAsySta}, we
have
\begin{align}
\mathcal{H}_{\nu}[f] - Q_m^{A}[f] =
\mathcal{O}(\omega^{-m-\frac{1}{r+1}}), \quad
\omega\rightarrow\infty,
\end{align}
where $Q_{m}^{A}[f]$ is defined as in \eqref{def:EndAsmMethII}.
\end{theorem}
\begin{proof}
It follows by combining Theorem \ref{thm:EndAsySta} and Theorem
\ref{lem:EstMomSta}.
\end{proof}

Consider an example with $f(x)= e^x$, $g(x)=x^2$, $\nu=2$ and
$[a,b]=[0,1]$. This example has a stationary point of order $r=1$ at
$\zeta=0$. In our computations, we rewrite $M_j(\zeta,\nu+k,\omega)$
as
\begin{align}
M_j(\zeta,\nu+k,\omega) &= \int_{0}^{1} x^j J_{\nu+k}(\omega x^2) dx
= \frac{1}{2} \int_{0}^{1} y^{\frac{j-1}{2}} J_{\nu+k}(\omega y) dy,
\nonumber
\end{align}
and then the last integral is calculated by \eqref{eq:MomLom} in the
appendix. In Figure \ref{fig:AsyStationary} we display the error of
$Q_{m}^{A}[f]$ in the left panel and the error scaled by
$\omega^{m+\frac{1}{2}}$ in the right panel. As expected, the
accuracy of $Q_{m}^{A}[f]$ improves as $\omega$ increases and the
decay rate of the error of $Q_{m}^{A}[f]$ is
$\mathcal{O}(\omega^{-m-\frac{1}{2}})$.

\begin{figure}[ht]
\centering
\includegraphics[width=7.2cm,height=6cm]{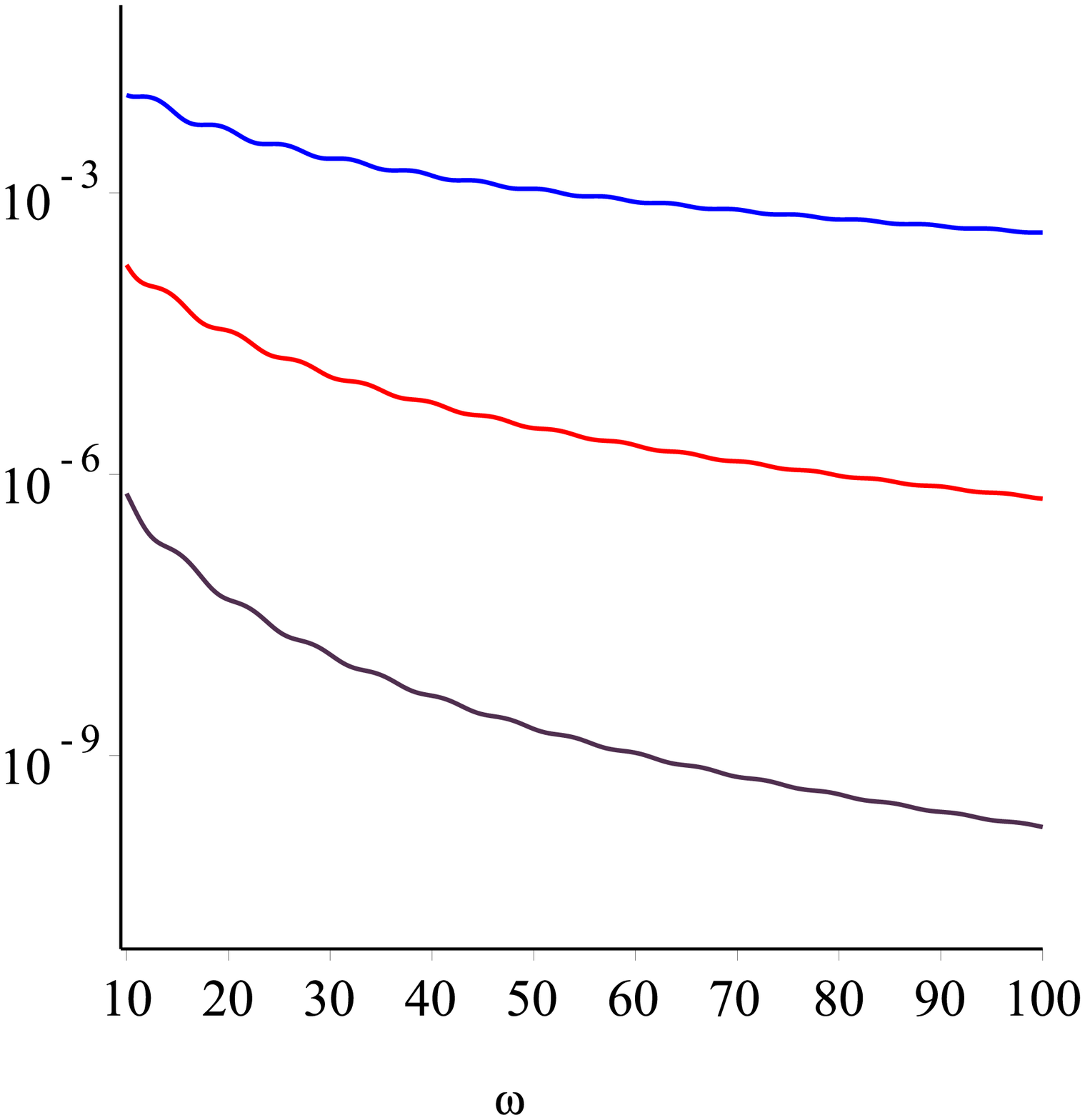}~
\includegraphics[width=7.2cm,height=6cm]{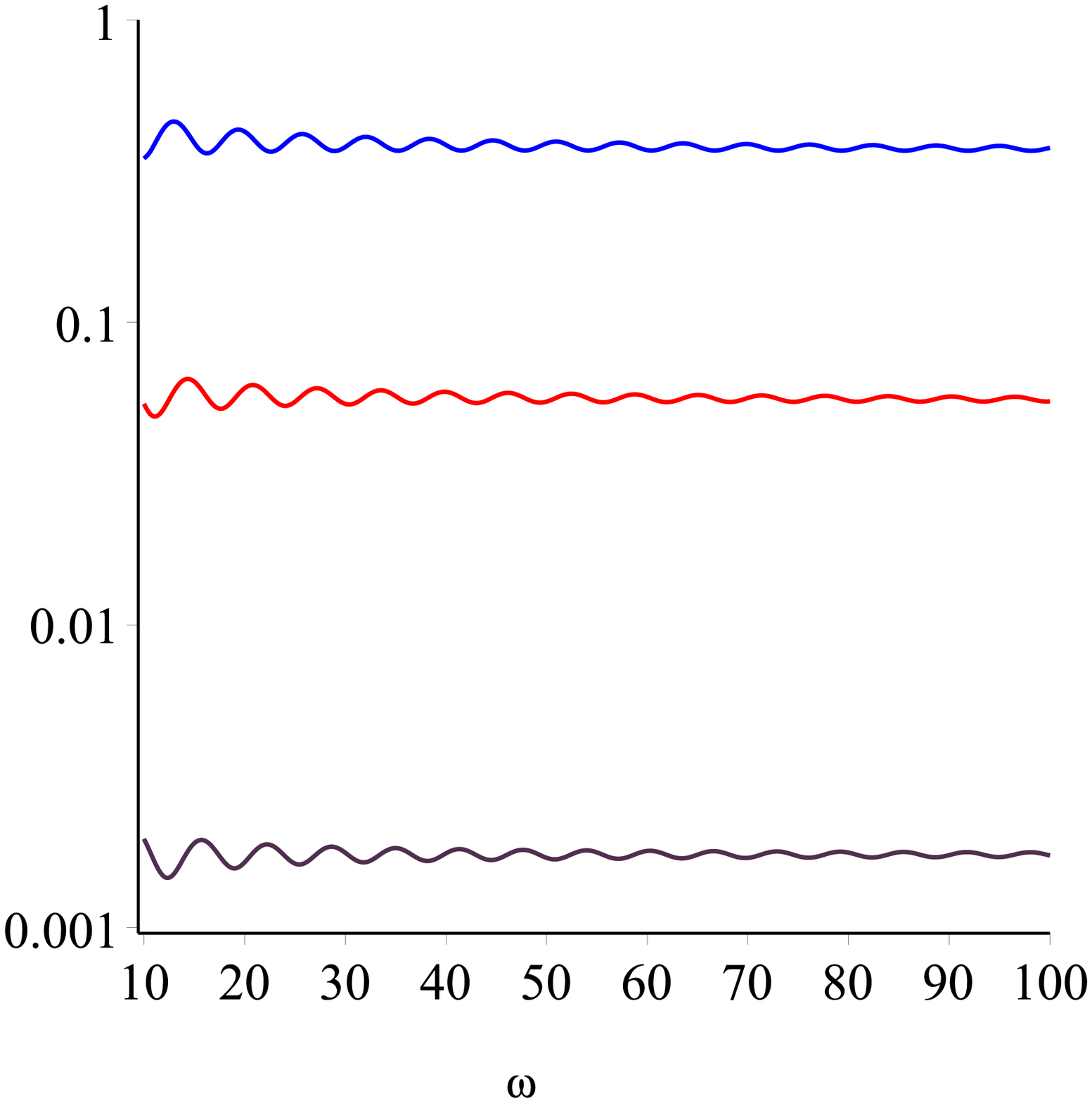}
\caption{Absolute error of $Q_m^{A}[f]$ (left) and the error scaled
by $\omega^{m+\frac{1}{2}}$ (right). 
The blue, red, violet curves correspond to $m=1,2,3$, respectively.}
\label{fig:AsyStationary}
\end{figure}

Now, we consider to construct a Filon-type method for the transform
in the presence of a stationary point at the left endpoint. To
achieve this, we introduce an auxiliary result.
\begin{lemma}
Assume that $g(x)\in C^{\infty}[a,b]$ and it has a stationary point
of type II and of order $r$ at $x=a$ and $g^{(r+1)}(x)\neq0$ for
$x\in[a,b]$. Let
\begin{align}
\widehat{\varphi}_k(x) = \left\{\begin{array}{cc}
                          {\displaystyle g{'}(x) g(x)^{\frac{k-r}{r+1}}},  &\mbox{if $g^{(r+1)}(x)>0$}, \\ [12pt]
                          {\displaystyle g{'}(x) (-g(x))^{\frac{k-r}{r+1}}},  &\mbox{if $g^{(r+1)}(x)<0$}.
                                        \end{array}
                                        \right. \nonumber
\end{align}
and set
\begin{align}
\widehat{\mathcal{E}} = \left\{ v(x)~\bigg|~ v(x) = \sum_{k=0}^{n-1}
c_k \widehat{\varphi}_k(x), ~~ c_k\in \mathbb{R} \right\}. \nonumber
\end{align}
Then, $\widehat{\mathcal{E}}$ forms an extended Chebyshev space.
\end{lemma}
\begin{proof}
Here we only consider the case of $g^{(r+1)}(x)>0$ for $x\in[a,b]$
and the other case $g^{(r+1)}(x)<0$ can be handled in a similar
manner. With this condition, we can deduce that $g^{(j)}(x)>0$ for
$x\in(a,b]$ and $j=1,\ldots,r$ and thus $g(x)$ is strictly
increasing in $[a,b]$. Setting $y^{r+1}=g(x)$, we have
\[
v(x) = \sum_{k=0}^{n-1} c_k \widehat{\varphi}_k(x) = \Lambda(x)
\sum_{k=0}^{n-1} c_k y^k,
\]
where $\Lambda(x) = g{'}(x) g(x)^{-\frac{r}{r+1}}$. Note that
$\Lambda(x)$ have a removable singularity at $x=a$ and we can easily
check that $\Lambda(a) = \lim_{x\rightarrow a}\Lambda(x) = (r+1)
\left[g^{(r+1)}(a)/(r+1)!\right]^{\frac{1}{r+1}}
>0$. Thus, we can conclude that $\Lambda(x)\in C^{\infty}[a,b]$ and
$\Lambda(x)>0$ for $x\in[a,b]$. The remaining details of the proof
is similar to that of Lemma \ref{lem:space} and hence are omitted.
\end{proof}

Let $a=x_0<x_1<\cdots<x_d=b$ be a set of distinct nodes with
multiplicities $m_0,m_1,\ldots,m_d$ and $\sum_{k=0}^{d} m_k = n$.
Suppose that $\widehat{p}(x) = \sum_{k=0}^{n-1} c_k
\widehat{\varphi}_k(x)\in \widehat{\mathcal{E}}$ is the unique
function which satisfies $\widehat{p}^{(j-1)}(x_k) = f^{(j-1)}(x_k)
$ for $j=1,\ldots,m_k$ and $k=0,\ldots,d$. The modified Filon-type
method is then defined by
\begin{align}\label{eq:ModFilonSta}
Q^{F}[f] := \mathcal{H}_{\nu}[\widehat{p}] = \sum_{k=0}^{n-1} c_k
\widehat{\mu}_{k,\nu}^{\omega}.
\end{align}
where $\widehat{\mu}_{k,\nu}^{\omega} =
\mathcal{H}_{\nu}[\widehat{\varphi}_k]$.

\begin{theorem}
Under the same assumptions as in Theorem \ref{thm:EndAsySta}, and
assume that $g^{(r+1)}(x)\neq0$ for $x\in[a,b]$. Let $m_d\geq m$ and
$m_0\geq m(r+1)$, then 
\begin{align}
\mathcal{H}_{\nu}[f] - Q^{F}[f] =
\mathcal{O}(\omega^{-m-\frac{1}{r+1}}).
\end{align}
\end{theorem}
\begin{proof}
Let $h(x) = f(x) - \widehat{p}(x) $ and we have
$\mathcal{H}_{\nu}[f]-Q^{F}[f] = \mathcal{H}_{\nu}[h]$. By using the
assumptions $m_d\geq m$ and $m_0\geq m(r+1)$ and the dependence
relation of $\sigma_k[f](x)$ on $f(x)$ and its derivatives, we have
for $k=0,\ldots,m-1$ that
\begin{align*}
\sigma_{k}[h](b) = \sigma_{k}[h]^{(j)}(a) = 0,
\end{align*}
where $j=0,\ldots,r$. The desired results follow from Theorem
\ref{thm:AsyErrorSta}.
\end{proof}

\begin{figure}[ht]
\centering
\includegraphics[width=7.2cm,height=6cm]{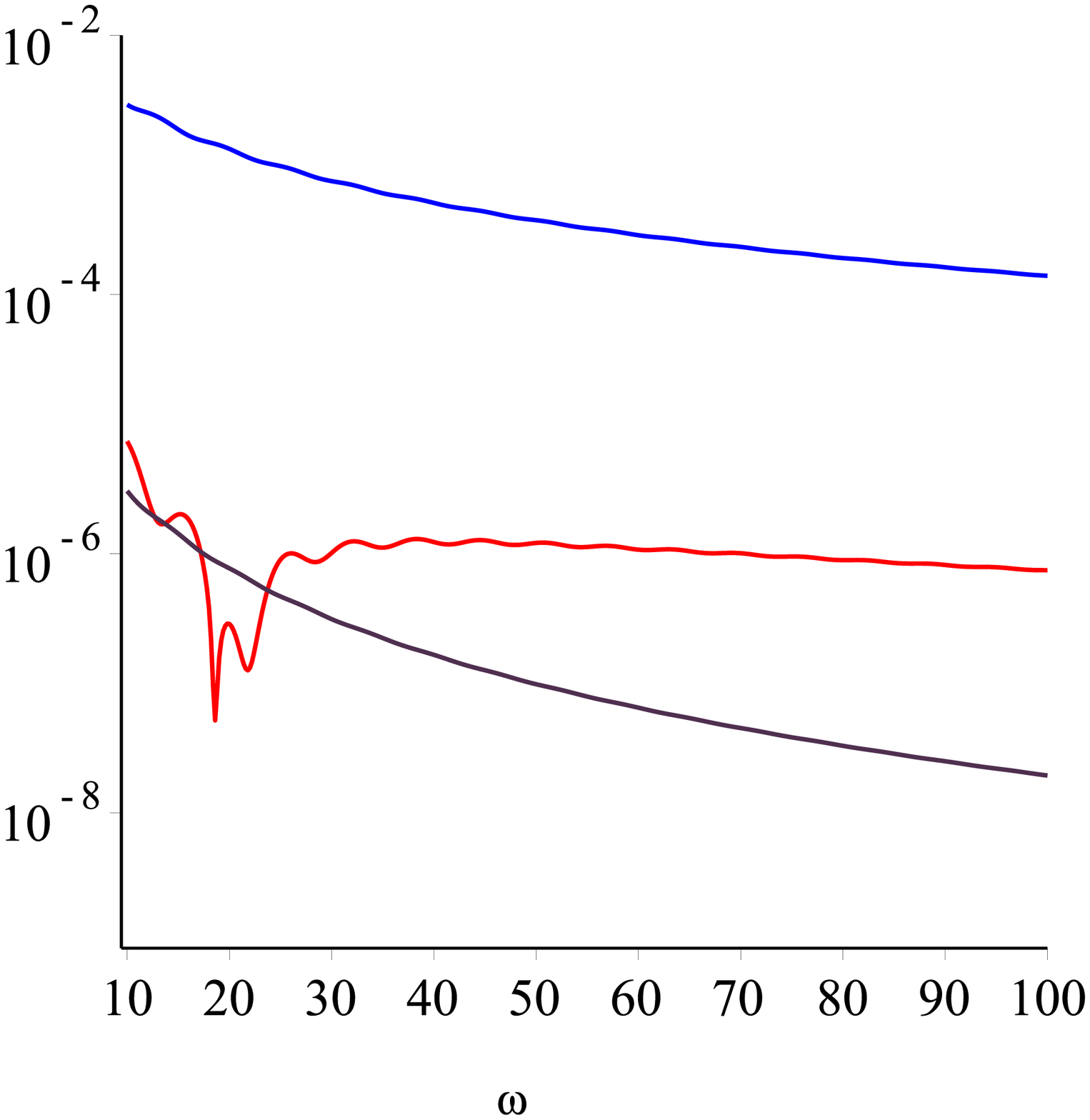}~
\includegraphics[width=7.2cm,height=6cm]{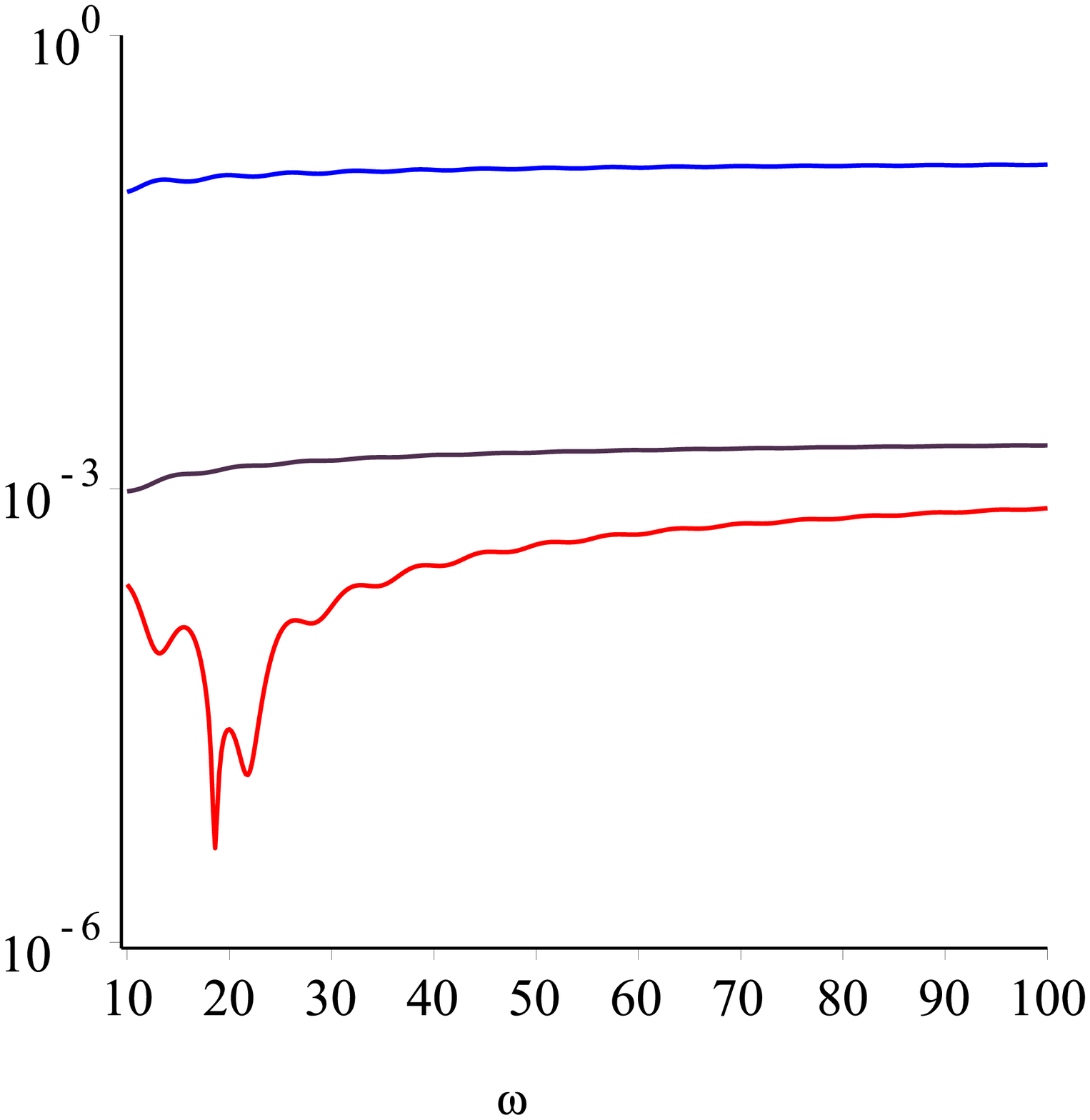}
\caption{The left panel shows the error of $Q^{F}[f]$ with nodes
$\{0,1\}$ and multiplicities $\{2,1\}$ (blue), $Q^{F}[f]$ with nodes
$\{0,\frac{1}{3},\frac{2}{3},1\}$ and multiplicities $\{2,1,1,1\}$
(red) and $Q^{F}[f]$ with nodes $\{0,1\}$ and multiplicities
$\{4,2\}$ (violet). The right panel shows the error of the first two
methods scaled by $\omega^{\frac{3}{2}}$ and the error of the last
method scaled by $\omega^{\frac{5}{2}}$.}
\label{fig:FilonStationary}
\end{figure}

We reconsider the example used in Figure \ref{fig:AsyStationary}. We
compare the accuracy of three modified Filon-type methods:
$Q^{F}[f]$ with nodes $\{0,1\}$ and multiplicities $\{2,1\}$,
$Q^{F}[f]$ with nodes $\{0,\frac{1}{3},\frac{2}{3},1\}$ and
multiplicities $\{2,1,1,1\}$ and $Q^{F}[f]$ with nodes $\{0,1\}$ and
multiplicities $\{4,2\}$. It is easy to see that the first two
Filon-type methods correspond to $m=1$ and therefore the decay rate
of their error is $\mathcal{O}(\omega^{-\frac{3}{2}})$. For the last
one, it corresponds to $m=2$ and therefore the decay rate of its
error is $\mathcal{O}(\omega^{-\frac{5}{2}})$. In our computations,
the moments $\{\widehat{\mu}_{k,\nu}^{\omega}\}_{k=0}^{n-1}$ are
evaluated by using Theorem \ref{thm:ModMomStatII} in the appendix.
Numerical results are presented in Figure \ref{fig:FilonStationary}.

\section{Extension}\label{sec:MultiHankel}
In this section we present several extensions of our results which
might be interesting in applications.
\subsection{Fourier-Bessel series}
The Fourier-Bessel series of a function $f: [0,b]\rightarrow
\mathbb{R}$ is defined by
\begin{align}
f(x) \sim \sum_{k=1}^{\infty} a_k J_{\nu}\left(\frac{j_{\nu,k}}{b}
x\right),
\end{align}
where $j_{\nu,k}$ is the $k$th positive root of $J_{\nu}(x)$ and the
coefficients are defined by
\begin{align}
a_k = \frac{2}{ [b J_{\nu+1}(j_{\nu,k})]^2} \int_{0}^{b} x f(x)
J_{\nu}\left(\frac{j_{\nu,k}}{b} x\right) dx.
\end{align}
This series is widely used in many areas such as the solution to
partial differential equations in cylindrical coordinate systems
(see \cite[Chap.~11]{poularikas2010transforms}). We observe that the
coefficients $\{ a_k \}_{k\geq0}$ are finite Hankel transforms and
therefore our results can be extended to study the convergence rate
and evaluation of the Fourier-Bessel series.

\subsection{Airy transform}
Consider the following integral
\begin{align}\label{eq:Airy}
I_{A}[f] = \int_{a}^{b} f(x) \mathrm{Ai}(-\omega x) dx,
\end{align}
where $\mathrm{Ai}(x)$ is the Airy function of the first kind and
$b>a\geq0$. A vector-valued asymptotic expansion of the integral
\eqref{eq:Airy} was presented in \cite{olver2007numerical} under the
condition that $a>0$. In the case of $a=0$, however, the
vector-valued expansion is no longer valid. In the following, we
will restrict our attention to the case $a=0$ and show that the
analysis presented in section \ref{sec:Stationary} can be easily
extended to deal with this problem.

By making use of the identity \cite[Eqn. 9.6.6]{olver2010nist}
\begin{align}
\mathrm{Ai}(-x) &= \frac{\sqrt{x}}{3} \left[
J_{\frac{1}{3}}\left(\frac{2}{3} x^{\frac{3}{2}} \right) +
J_{-\frac{1}{3}}\left(\frac{2}{3} x^{\frac{3}{2}} \right) \right],
\end{align}
it follows
\begin{align}
I_{A}[f] = \frac{\sqrt{\omega}}{3} \left[ \int_{0}^{b} \sqrt{x} f(x)
J_{\frac{1}{3}}\left(\frac{2}{3}\omega^{\frac{3}{2}} x^{\frac{3}{2}}
\right) dx + \int_{0}^{b} \sqrt{x} f(x)
J_{-\frac{1}{3}}\left(\frac{2}{3}\omega^{\frac{3}{2}}
x^{\frac{3}{2}} \right) dx  \right].  \nonumber
\end{align}
By setting $x=t^2$ and $\widehat{\omega} =
\frac{2}{3}\omega^{\frac{3}{2}}$, we further arrive at
\begin{align}\label{eq:AiryTran}
I_{A}[f] = \frac{2\sqrt{\omega}}{3} \left[ \int_{0}^{\sqrt{b}} t^2
f(t^2) J_{\frac{1}{3}}\left( \widehat{\omega} t^3 \right) dt +
\int_{0}^{\sqrt{b}} t^2 f(t^2)
J_{-\frac{1}{3}}\left(\widehat{\omega} t^3 \right) dt \right].
\end{align}
Clearly, we see that both integrals on the right-hand side of
\eqref{eq:AiryTran} are finite Hankel transforms involving the
oscillator $g(t) = t^3$. Therefore, the result of Corollary
\ref{cor:AsyStaSim} can be used directly to obtain an asymptotic
expansion of $I_{A}[f]$.

\subsection{Multivariate Hankel transform}
Let $\Omega$ be a connected and bounded region, we consider
\begin{align}\label{eq:MultiHankel}
\mathcal{H}_{\nu}[f] = \int_{\Omega} f(\vec{x}) J_{\nu}(\omega
g(\vec{x})) dV,
\end{align}
where $f,g:\mathbb{R}^{d}\rightarrow\mathbb{R}$ be sufficiently
smooth functions. Integrals of this form may arise in some
applications such as the crystallography problem
\cite{mcclure1990asymptotic}.

Asymptotic analysis of the multivariate Hankel transform is much
more involved than the univariate case. Generally speaking, the
asymptotic behavior of the transform depends not only on the zeros
and stationary points, i.e., $\vec{x}$ where $g(\vec{x})=0$ and
$\nabla g(\mathbf{x})=0$ respectively, but also points of resonance,
i.e., $\vec{x}$ where $\nabla g(\mathbf{x})$ is orthogonal to the
boundary. In the following we restrict our attention to the ideal
case where the oscillator $g(\vec{x})$ has no zeros, stationary
points and resonance points for all $\vec{x}\in\Omega$.
\begin{theorem}\label{thm:MultiAsyExp}
Assume that $g(\mathbf{x})$ has no zero, stationary and resonance
points for all $\vec{x}\in\Omega$. Then, for
$\omega\rightarrow\infty$, it is true that
\begin{align}\label{eq:MultiAsyExp}
\mathcal{H}_{\nu}[f] \sim - \sum_{k=1}^{\infty}
\frac{1}{(-\omega)^{k}} \int_{\partial\Omega} \frac{\nabla
g(\mathbf{x})\cdot \vec{n}}{|\nabla g(\mathbf{x})|^2}
\sigma_{k-1}(\mathbf{x}) J_{\nu+k}(\omega g(\mathbf{x})) dS,
\end{align}
where $\vec{n}$ is the unit outward normal to $\partial \Omega$ and
$\sigma_k(\mathbf{x})$ are defined by
\begin{align}
\sigma_0(\mathbf{x}) &= f(\mathbf{x}), \nonumber \\
\sigma_k(\mathbf{x}) &= g(\mathbf{x})^{\nu+k}  \nabla \cdot \left[
\frac{\nabla g(\mathbf{x})}{|\nabla g(\mathbf{x})|^2
g(\mathbf{x})^{\nu+k}} \sigma_{k-1}(\mathbf{x}) \right] , \quad
k\geq 1.
\end{align}
\end{theorem}
\begin{proof}
We first define
\begin{align}
\mathbf{F}(\mathbf{x}) = \frac{\nabla g(\mathbf{x})}{|\nabla
g(\mathbf{x})|^2 g(\mathbf{x})^{\nu+m+1}} \sigma_{m}(\mathbf{x}),
\quad \Phi(\mathbf{x}) =  g(\mathbf{x})^{\nu+m+1} J_{\nu+m+1}(\omega
g(\mathbf{x})), \nonumber
\end{align}
where $m\geq0$. Clearly, $\mathbf{F(x)}$ is a differentiable vector
function and $\Phi(\mathbf{x})$ is a differentiable scalar function.
With the help of the divergence operator
\cite[p.~1051]{gradshteyn2007table} and \eqref{eq:DE}, we have
\begin{align}
\nabla \cdot [\mathbf{F(x)} \Phi(\mathbf{x})] & = \Phi(\mathbf{x})
(\nabla \cdot \mathbf{F(x)}) + \mathbf{F(x)} \cdot \nabla
\Phi(\mathbf{x}) \nonumber \\
& = \Phi(\mathbf{x}) (\nabla \cdot \mathbf{F(x)}) + \omega
\sigma_m(\mathbf{x}) J_{\nu+m}(\omega g(\mathbf{x})), \nonumber
\end{align}
and therefore
\begin{align}
\int_{\Omega} \sigma_{m}(\mathbf{x}) J_{\nu+m}(\omega g(\mathbf{x}))
dV &= \frac{1}{\omega} \int_{\Omega} \nabla \cdot [\mathbf{F(x)}
\Phi(\mathbf{x})] dV - \frac{1}{ \omega}
\int_{\Omega} \Phi(\mathbf{x}) (\nabla \cdot \mathbf{F(x)}) dV \nonumber \\
&= \frac{1}{\omega} \int_{\partial\Omega} \Phi(\mathbf{x})
(\mathbf{F(x)} \cdot \mathbf{n} ) dS - \frac{1}{ \omega}
\int_{\Omega} \Phi(\mathbf{x}) (\nabla \cdot \mathbf{F(x)}) dV  \nonumber \\
&= \frac{1}{\omega} \int_{\partial\Omega} \frac{\nabla g(\mathbf{x})
\cdot \mathbf{n}}{|\nabla g(\mathbf{x})|^2} \sigma_m(\mathbf{x})
J_{\nu+m+1}(\omega
g(\mathbf{x})) dS  \nonumber \\
&~~~~~ - \frac{1}{ \omega} \int_{\Omega} \sigma_{m+1}(\mathbf{x})
J_{\nu+m+1}(\omega g(\mathbf{x})) dV, \nonumber
\end{align}
where we have used Gauss's divergence theorem in the second step.
Iterating this process from $m=0$ to $m=s$, we arrive at
\begin{align}
\mathcal{H}_{\nu}[f] &= - \sum_{k=1}^{s} \frac{1}{(-\omega)^{k}}
\int_{\partial\Omega} \frac{\nabla g(\mathbf{x})\cdot
\vec{n}}{|\nabla g(\mathbf{x})|^2} \sigma_{k-1}(\mathbf{x})
J_{\nu+k}(\omega g(\mathbf{x})) dS \nonumber \\
&~~~~~~~~~~ + \frac{1}{(-\omega)^{s}} \int_{\Omega}
\sigma_{s}(\mathbf{x}) J_{\nu+s}(\omega g(\mathbf{x})) dV. \nonumber
\end{align}
Letting $s\rightarrow\infty$ gives the desired result.
\end{proof}

The expansion \eqref{eq:MultiAsyExp} is a multivariate
generalization of \eqref{eq:AsymCaseOne} and it shows that the value
of $\mathcal{H}_{\nu}[f]$ over the region $\Omega$ can be
asymptotically determined by integrals on the boundaries of
$\Omega$. When $\Omega$ has piecewise smooth boundaries such as the
$d$-dimensional simplex, those integrals on the right hand side of
\eqref{eq:MultiAsyExp} can be expanded repeatedly by using
integration by parts until they are expressed by using point values
and derivatives at the vertices. We expect that the result of
Theorem \ref{thm:MultiAsyExp} might play an important role in the
asymptotic and numerical study of multivariate Hankel transform.

\section{Conclusions}\label{sec:con}
In the present paper, we presented a unified framework for
asymptotic analysis and the computation of the finite Hankel
transform. Asymptotic expansions of the transform were established
in the presence of critical points, e.g., zeros and stationary
points, and subsequently two schemes for computing the transform
numerically were developed. We provided a rigorous error analysis
and obtained asymptotic error estimates for these two methods.

We also showed that the analysis given here can be generalized to
the multivariate setting. 
This result is only the first step in a long journey and much effort
will be needed to understand the asymptotic and numerical methods of
the multivariate Hankel transform, especially in the presence of
zeros, stationary and resonance
points. 

\appendix
\section[Appendix]{The moments for the simplest oscillator $g(x)=x$}\label{sec:Mom}
For the special case where $g(x)=x$ and $[a,b]=[0,1]$, we consider
the formula for the standard moments of the Hankel transform, i.e.,
\begin{align}\label{lem:Mom}
\mathcal{H}_{\nu}[x^{\mu}] = \int_{0}^{1} x^{\mu} J_{\nu}(\omega
x)dx,
\end{align}
where $\omega>0$ and $\Re(\mu+\nu)>-1$.

From \cite[p.~676]{gradshteyn2007table}, we know that
\begin{align}\label{eq:MomLom}
\mathcal{H}_{\nu}[x^{\mu}] &= \frac{2^{\mu}
\Gamma(\frac{\nu+\mu+1}{2})}{\omega^{\mu+1}
\Gamma(\frac{\nu-\mu+1}{2})} + \frac{(\mu+\nu-1) J_{\nu}(\omega)
S_{\mu-1,\nu-1}(\omega) - J_{\nu-1}(\omega)
S_{\mu,\nu}(\omega)}{\omega^{\mu}},
\end{align}
where $\Gamma(z)$ is the gamma function and $S_{\mu,\nu}(z)$ is the
Lommel function. As $z\rightarrow\infty$, the Lommel function admits
the following asymptotic expansion
\cite[p.~947]{gradshteyn2007table}
\begin{align}\label{eq:LomAsy}
S_{\mu,\nu}(z) \sim z^{\mu-1} \sum_{m=0}^{\infty}(-1)^m
\left(\frac{1-\mu+\nu}{2} \right)_m  \left(\frac{1-\mu-\nu}{2}
\right)_m \left(\frac{z}{2} \right)^{-2m}.
\end{align}
In fact, a full asymptotic expansion of $\mathcal{H}_{\nu}[x^{\mu}]$
can be derived by combining \eqref{eq:MomLom} and \eqref{eq:LomAsy}
together with the asymptotic expansion of the Bessel function.

As a direct consequence of \eqref{eq:MomLom}, we have the following
estimate.
\begin{lemma}\label{lem:EstimMom}
Assume that $\Re(\mu+\nu)>-1$. Then, for $\omega\rightarrow\infty$,
\begin{align}\label{def:MomAsym1}
\mathcal{H}_{\nu}[x^{\mu}] = \left\{
                    \begin{array}{ll}
                     \mathcal{O}( \omega^{-\min\left\{\mu,\frac{1}{2}\right\}-1} ),   & \hbox{if $\mu-\nu\neq1,3,5,\ldots$,} \\
                     [10pt]
                     \mathcal{O}( \omega^{-\frac{3}{2}} ),   & \hbox{if $\mu-\nu = 1,3,5,\ldots$.}
                    \end{array}
                  \right.
\end{align}
\end{lemma}
\begin{proof}
When $\mu-\nu\neq1,3,5,\ldots$, the desired estimate follows by
combining \eqref{eq:MomLom} and \eqref{eq:LomAsy} together with the
fact that $J_{\nu}(x) = \mathcal{O}(x^{-\frac{1}{2}})$ as
$x\rightarrow\infty$. When $\mu-\nu=1,3,5,\ldots$, note that the
first term in \eqref{eq:MomLom} vanishes and therefore
\[
\mathcal{H}_{\nu}[x^{\mu}] = \frac{(\mu+\nu-1) J_{\nu}(\omega)
S_{\mu-1,\nu-1}(\omega) - J_{\nu-1}(\omega)
S_{\mu,\nu}(\omega)}{\omega^{\mu}}.
\]
The desired estimate follows by combining \eqref{eq:LomAsy} and the
fact that $J_{\nu}(x) = \mathcal{O}(x^{-\frac{1}{2}})$.
\end{proof}

\section[Appendix]{The evaluation of the modified moments}
We present here a recurrence relation for $\mu_{k,\nu}^{\omega}$
which are defined in \eqref{eq:ModifiedFilon}.
\begin{theorem}\label{thm:ModMomRecRel}
Assume that $g{'}(x)\neq0$ for $x\in[a,b]$. For $k\geq0$, we have
\begin{align}\label{eq:ModMomRecRel}
\mu_{k+2,\nu}^{\omega} &= \frac{\nu^2 - (k+1)^2}{\omega^2}
\mu_{k,\nu}^{\omega} + \mathcal{R}_{k,\nu}^{\omega},
\end{align}
where
\begin{align}
\mathcal{R}_{k,\nu}^{\omega} &=  \frac{1}{\omega} \left[ g(a)^{k+2}
\frac{J_{\nu-1}(\omega g(a)) - J_{\nu+1}(\omega g(a))}{2} -
g(b)^{k+2} \frac{J_{\nu-1}(\omega g(b)) - J_{\nu+1}(\omega g(b))}{2} \right] \nonumber \\
&~~~~~~~~~~~ + \frac{k+1}{\omega^2} \bigg[ g(b)^{k+1} J_{\nu}(\omega
g(b)) - g(a)^{k+1} J_{\nu}(\omega g(a)) \bigg]. \nonumber
\end{align}
\end{theorem}
\begin{proof}
By setting $z = \omega g(x)$ and $w(z) = J_{\nu}(z)$. Recall the
Bessel's equation
\begin{align}
z^2 \frac{d^2 w}{dz^2} + z \frac{dw}{dz} + (z^2 - \nu^2)w = 0,
\nonumber
\end{align}
it follows that
\begin{align}\label{eq:RecStepOne}
\mu_{k+2,\nu}^{\omega} 
&= \frac{1}{\omega^2} \int_{a}^{b} g{'}(x) g(x)^{k} z^2
J_{\nu}(z)dx \nonumber \\
&= \frac{1}{\omega^2} \int_{a}^{b} g{'}(x) g(x)^{k} \left( \nu^2
w(z) - z \frac{dw}{dz} - z^2 \frac{d^2w}{dz^2} \right) dx
\nonumber \\
&= \frac{\nu^2}{\omega^2} \mu_{k,\nu}^{\omega} - \frac{1}{\omega^2}
\int_{a}^{b} g{'}(x) g(x)^{k} z \frac{dw}{dz} dx -
\frac{1}{\omega^2} \int_{a}^{b} g{'}(x) g(x)^{k} z^2
\frac{d^2w}{dz^2} dx.
\end{align}
For the first integral on the right-hand side of
\eqref{eq:RecStepOne}, using integration by part, we have
\begin{align}\label{eq:MuPartOne}
\int_{a}^{b} g{'}(x) g(x)^{k} z \frac{dw}{dz} dx &= \int_{a}^{b}
g(x)^{k+1} \frac{d}{dx}\left( J_{\nu}(\omega g(x)\right) dx
\nonumber \\
&= g(b)^{k+1} J_{\nu}(\omega g(b)) - g(a)^{k+1} J_{\nu}(\omega g(a))
\nonumber \\
&~~~~~ - (k+1) \int_{a}^{b} g{'}(x) g(x)^k J_{\nu}(\omega g(x)) dx
\nonumber \\
&= g(b)^{k+1} J_{\nu}(\omega g(b)) - g(a)^{k+1} J_{\nu}(\omega g(a))
- (k+1) \mu_{k,\nu}^{\omega}.
\end{align}
Similarly, for the second integral on the right-hand side of
\eqref{eq:RecStepOne}, we have
\begin{align}\label{eq:MuPartTwo}
\int_{a}^{b} g{'}(x) g(x)^{k} z^2 \frac{d^2w}{dz^2} dx &= \omega
\int_{a}^{b} g(x)^{k+2}
\frac{d}{dx}\left( \frac{dw}{dz} \right) dx \nonumber \\
&= \omega \left[ g(x)^{k+2} \frac{dw}{dz} \right]_{a}^{b} - \omega
(k+2) \int_{a}^{b} g{'}(x) g(x)^{k+1} \frac{dw}{dz} dx.
\end{align}
Applying the following formula $\frac{d}{dz} J_{\nu}(z) =
\frac{1}{2} ( J_{\nu-1}(z) - J_{\nu+1}(z) )$ \cite[Eqn.
10.6.1]{olver2010nist} and using integration by part to the last
integral in \eqref{eq:MuPartTwo}, we arrive at
\begin{align}\label{eq:MuPartThree}
\int_{a}^{b} g{'}(x) g(x)^{k} z^2 \frac{d^2w}{dz^2} dx &= \omega
\left[ g(b)^{k+2} \frac{J_{\nu-1}(\omega g(b)) -
J_{\nu+1}(\omega g(b))}{2} \right. \nonumber \\
&~~~~~~ \left. - g(a)^{k+2} \frac{J_{\nu-1}(\omega g(a)) -
J_{\nu+1}(\omega g(a))}{2} \right] \nonumber \\
&~~~~~~ - (k+2) \left[ g(b)^{k+1} J_{\nu}(\omega g(b)) - g(a)^{k+1}
J_{\nu}(\omega g(a)) \right] \nonumber \\
&~~~~~~ + (k+2)(k+1) \mu_{k,\nu}^{\omega}.
\end{align}
Substituting \eqref{eq:MuPartOne} and \eqref{eq:MuPartThree} into
\eqref{eq:RecStepOne}, we obtain the desired result.
\end{proof}

When using Theorem \ref{thm:ModMomRecRel}, two initial values
$\mu_{0,\nu}^{\omega}$ and $\mu_{1,\nu}^{\omega}$ are required. We
proceed as follows: By setting $y=g(x)$, we obtain
\begin{align}
\mu_{k,\nu}^{\omega} &= \int_{g(a)}^{g(b)} y^k J_{\nu}(\omega y) dy
= \int_{0}^{g(b)} y^k J_{\nu}(\omega y) dy - \int_{0}^{g(a)} y^k
J_{\nu}(\omega y) dy.  \nonumber
\end{align}
For these two integrals in the last equality, by scaling their
intervals into $[0,1]$, we have
\begin{align}\label{eq:BesselMomF}
\mu_{k,\nu}^{\omega} &= g(b)^{k+1} \int_{0}^{1} y^k J_{\nu}(\omega
g(b) y) dy - g(a)^{k+1} \int_{0}^{1} y^k J_{\nu}(\omega g(a) y) dy.
\end{align}
Finally, these two integrals on the right-hand side of
\eqref{eq:BesselMomF} can be evaluated by \eqref{eq:MomLom}. We
remark that, under the condition $ |\nu|\leq\omega$, the recurrence
relation \eqref{eq:ModMomRecRel} is forward stable when $k\leq
\sqrt{|\nu|^2+\omega^2}-1$.

The above result can be extended to the modified moments
$\widehat{\mu}_{k,\nu}^{\omega}$.
\begin{theorem}\label{thm:ModMomStatII}
Assume that $g(x)$ has a stationary point of type II and of order
$r$ at $x=a$. Furthermore, we assume that $g^{(r+1)}(x)\neq0$ for
$x\in[a,b]$. For $k\geq0$, we have
\begin{align}\label{eq:ModMomRecRelII}
\widehat{\mu}_{k+2r+2,\nu}^{\omega} = \frac{1}{\omega^2} \left[
\nu^2 - \left(\frac{k+1}{r+1} \right)^2 \right]
\widehat{\mu}_{k,\nu}^{\omega} +
\widehat{\mathcal{R}}_{k,\nu}^{\omega} \times \left\{
                    \begin{array}{ll}
                     1, & \hbox{if $g^{(r+1)}(x)>0$,} \\
                     [10pt]
                     e^{i(\nu+1)\pi},  & \hbox{if $g^{(r+1)}(x)<0$,}
                    \end{array}
                  \right.
\end{align}
where
\begin{align}
\widehat{\mathcal{R}}_{k,\nu}^{\omega} &= \left( \frac{k+1}{r+1}
\right) \frac{J_{\nu}(\omega g(b))}{\omega^2} g(b)^{\frac{k+1}{r+1}}
- \frac{J_{\nu-1}(\omega g(b)) - J_{\nu+1}(\omega g(b))}{2\omega}
g(b)^{\frac{k+1}{r+1}+1}. \nonumber
\end{align}
\end{theorem}
\begin{proof}
The proof is similar to that of Theorem \ref{thm:ModMomRecRel}. We
omit the details.
\end{proof}

When using Theorem \ref{thm:ModMomStatII}, the first $2r+2$ initial
values $\{\widehat{\mu}_{k,\nu}^{\omega}\}_{k=0}^{2r+1}$ are
required. For example, in the case of $g^{(r+1)}(x)>0$ where
$x\in[a,b]$, by setting $t=g(x)/g(b)$, we have
\begin{align}
\widehat{\mu}_{k,\nu}^{\omega} = g(b)^{\frac{k+1}{r+1}} \int_{0}^{1}
t^{\frac{k-r}{r+1}} J_{\nu}(\omega g(b) t) dt. \nonumber
\end{align}
The integral on the right-hand side can be calculated by using
\eqref{eq:MomLom}. Moreover, under the condition $|\nu|\leq \omega$,
the recurrence relation \eqref{eq:ModMomRecRelII} is forward stable
when $k\leq (r+1)(\sqrt{|\nu|^2+\omega^2}-1)$.

\section[Appendix]{Dependence of $\widehat{\sigma}_k[f](x)$ on $f(x)$ at the stationary point}
Let $\widehat{\sigma}_k[f](x)$ be defined as in
\eqref{eq:FulSigSta}. We aim at studying the dependence of
$\widehat{\sigma}_k[f](x)$ on $f(x)$ and its derivatives at the
stationary point $x=\zeta$. We start from a simple lemma.
\begin{lemma}\label{lem:RatTaySer}
Suppose that
\begin{align}
N(t) = \sum_{j=0}^{\infty} a_j t^j, \quad  D(t) =
\sum_{j=0}^{\infty} b_j t^j, \nonumber
\end{align}
where $b_0\neq0$. Furthermore, suppose that
\[
R(t) = \frac{N(t)}{D(t)} = \sum_{j=0}^{\infty} c_j t^j.
\]
Then, each $c_j$ is a linear combination of $a_0,a_1,\ldots,a_j$.
\end{lemma}
\begin{proof}
Note that $N(t)=D(t)R(t)$,  we have
\begin{align}\label{eq:CauchyProd}
a_j = \sum_{k=0}^{j} c_k b_{j-k}, \quad j\geq0.
\end{align}
Let $\mathbf{a} = [a_0,\cdots,a_j]^T$ and $\mathbf{c} =
[c_0,\cdots,c_j]^T$ where $\cdot^{T}$ denotes the transpose. We can
rewrite \eqref{eq:CauchyProd} as
\[
\mathbf{a} = A \mathbf{c},
\]
where
\begin{align}
A = \left(
      \begin{array}{cccc}
        b_0 &  &  &  \\
        b_1 & b_0 &  &  \\
        \vdots & \vdots & \ddots &  \\
        b_j & b_{j-1} & \cdots & b_0 \\
      \end{array}
    \right). \nonumber
\end{align}
It is clear to see that $A$ is a lower triangular matrix and is
invertible. Note that $\mathbf{c} = A^{-1} \mathbf{a}$ and $A^{-1}$
is also a lower triangular matrix. The desired result follows.
\end{proof}

We now prove the following result.
\begin{theorem}\label{thm:DepRelSigma}
Let $f(x),g(x)$ be analytic functions inside a neighborhood of the
interval $[a,b]$. Moreover, assume that $g(x)$ has a stationary
point of type II and of order $r$ at $\zeta\in[a,b]$ and
$g{'}(x)\neq0$ for $x\in[a,b]\backslash\{\zeta\}$. Then, the value
$\widehat{\sigma}_k[f]^{(j)}(\zeta)$ is a linear combination of the
values
$f^{(r+1+j)}(\zeta),f^{(r+2+j)}(\zeta),\ldots,f^{(k(r+1)+j)}(\zeta)$
for all $k,j\geq0$.
\end{theorem}
\begin{proof}
With the assumption on $g(x)$, we have $g^{(j)}(\zeta) = 0$ for
$j=0,\ldots,r$ and $g^{(r+1)}(\zeta)\neq0$. By Taylor expansion, we
have
\begin{align}\label{eq:TaylorSeries}
f(x) = \sum_{j=0}^{\infty} \frac{f_j}{j!} (x-\zeta)^j, \quad g(x) =
\sum_{j=r+1}^{\infty} \frac{g_j}{j!} (x-\zeta)^j,
\end{align}
where $f_j = f^{(j)}(\zeta)$ and $g_j = g^{(j)}(\zeta)$.

We will prove the assertion by induction. For $k=0$, the assertion
is true since $\widehat{\sigma}_0[f](x)=f(x)$. Now we suppose the
assertion is true for some $k$, i.e.,
$\widehat{\sigma}_k[f]^{(j)}(\zeta)$ is a linear combination of
$f_{r+1},f_{r+2},\ldots,f_{k(r+1)+j}$ for all $j\geq0$. We will
prove it is true for $k+1$. By the Taylor expansion of
$\widehat{\sigma}_k[f](x)$ about $x=\zeta$, we have
\begin{align}\label{eq:SigmaKTay}
\widehat{\sigma}_k[f](x) = \sum_{j=0}^{\infty} \frac{D_{k,j}}{j!}
(x-\zeta)^j,
\end{align}
where $D_{k,j} = \widehat{\sigma}_k[f]^{(j)}(\zeta)$. In the
following, we consider the derivation of the Taylor expansion of
$\widehat{\sigma}_{k+1}[f](x)$ about $x=\zeta$ from
\eqref{eq:SigmaKTay}. By combining \eqref{eq:FulSigSta} and
\eqref{eq:progression}, we easily get
\begin{align}\label{eq:AppendSigmaK}
\widehat{\sigma}_{k+1}[f](x) &=
\frac{d}{dx}\left[\frac{\widehat{\sigma}_k[f](x) -
T_{r}[\widehat{\sigma}_k[f]](x,\zeta) }{g{'}(x)} \right] - (\nu+k+1)
\frac{\widehat{\sigma}_k[f](x) -
T_{r}[\widehat{\sigma}_k[f]](x,\zeta) }{g(x)}.
\end{align}
For the term inside the bracket, using \eqref{eq:TaylorSeries} and
\eqref{eq:SigmaKTay} and expanding it about $x=\zeta$ we obtain
\begin{align}\label{eq:T4}
\frac{ \widehat{\sigma}_k[f](x) -
T_{r}[\widehat{\sigma}_k[f]](x,\zeta) }{g{'}(x)} &=
\frac{\displaystyle (x-\zeta)\sum_{j=0}^{\infty}
\frac{D_{k,j+r+1}}{(j+r+1)!} (x-\zeta)^{j} }{\displaystyle
\sum_{j=0}^{\infty} \frac{g_{j+r+1}}{(j+r)!} (x-\zeta)^{j}}
\nonumber \\
& = (x-\zeta) \sum_{j=0}^{\infty} \widetilde{E}_{j} (x-\zeta)^j.
\end{align}
By Lemma \ref{lem:RatTaySer}, we see that each $\widetilde{E}_j$ is
a linear combination of the values $D_{k,r+1},\ldots,D_{k,r+1+j}$.
Following the same line, for the second term in
\eqref{eq:AppendSigmaK}, we have
\begin{align}\label{eq:T5}
\frac{\displaystyle \widehat{\sigma}_k[f](x) -
T_{r}[\widehat{\sigma}_k[f]](x,\zeta) }{g(x)} &= \frac{\displaystyle
\sum_{j=0}^{\infty} \frac{D_{k,j+r+1}}{(j+r+1)!} (x-\zeta)^j
}{\displaystyle \sum_{j=0}^{\infty}\frac{g_{r+1+j}}{(r+1+j)!}
(x-\zeta)^j} = \sum_{j=0}^{\infty} \widehat{E}_j (x-\zeta)^j,
\end{align}
Again, by Lemma \ref{lem:RatTaySer} we see that each $\widehat{E}_j$
is a linear combination of $D_{k,r+1},\ldots,D_{k,r+1+j}$.
Substituting \eqref{eq:T4} and \eqref{eq:T5} into
\eqref{eq:AppendSigmaK}, we get
\begin{align}
\widehat{\sigma}_{k+1}[f](x) = \sum_{j=0}^{\infty} \left[ (j+1)
\widehat{E}_j - (\nu+k+1) \widehat{E}_j \right] (x-\zeta)^j,
\nonumber
\end{align}
or equivalently,
\[
\frac{\sigma_{k+1}[f]^{(j)}(\zeta)}{j!} = (j+1) \widetilde{E}_j -
(\nu+k+1) \widehat{E}_j, \quad j\geq0.
\]
Clearly, we can deduce that $\sigma_{k+1}[f]^{(j)}(\zeta)$ is a
linear combination of $D_{k,r+1},\ldots,D_{k,r+1+j}$ and thus it is
a linear combination of $f_{r+1},\ldots,f_{(k+1)(r+1)+j}$. This
proves the case $k+1$. By induction, the assertion is true for all
$k\geq0$. This complete the proof.
\end{proof}

\section*{Acknowledgements}
This work was supported by the National Natural Science Foundation
of China under grant 11671160. The author would like to thank Daan
Huybrechs for supporting his visit to University of Leuven where
parts of this work were developed. The author also thanks Arieh
Iserles for valuable suggestions on the modified moments and
Shuhuang Xiang for helpful discussions.

\bibliographystyle{abbrv}
\bibliography{hankel}



\end{document}